\documentclass[12pt]{elsarticle}

\usepackage[utf8]{inputenc}
\usepackage{xcolor}
\usepackage{multicol}
\usepackage[table]{xcolor}
\usepackage{subfigure}
\usepackage[utf8]{inputenc}
\usepackage{graphicx}
\usepackage{titlesec}
\usepackage[bookmarks,breaklinks,colorlinks=true,allcolors=blue]{hyperref}
\usepackage{listings}
\usepackage{inconsolata}
\usepackage{float}
\usepackage{arydshln}
\usepackage{enumitem}
\usepackage[normalem]{ulem}

\usepackage{geometry}
\usepackage{amsmath}
\usepackage{parskip}
\usepackage[official]{eurosym}
\usepackage{todonotes}
\usepackage{csquotes}
\usepackage[]{amsthm}
\usepackage[]{mathtools}
\usepackage[]{bm}
\usepackage[]{thmtools}
\usepackage{amssymb}
\usepackage{tikz}
\usetikzlibrary{automata,positioning,arrows} 
\titleformat{\section}[block]{\normalfont\Large\bfseries}{\thesection.}{.5em}{\Large}

\definecolor{USred}{cmyk}{0,1.00,0.65,0.34}

\newcommand{\cabeceraespecial}{%
    
    \normalfont\bfseries}
\declaretheoremstyle[
    spaceabove=\medskipamount,
    spacebelow=\medskipamount,
    headfont=\cabeceraespecial,
    notefont=\cabeceraespecial,
    notebraces={(}{)},
    bodyfont=\normalfont\itshape,
    postheadspace=1em,
    numberwithin=section,
    headindent=0pt,
    headpunct={.}
    ]{importante}
\declaretheoremstyle[
    spaceabove=\medskipamount,
    spacebelow=\medskipamount,
    headfont=\cabeceraespecial,
    notefont=\cabeceraespecial,
    notebraces={(}{)},
    bodyfont=\normalfont\itshape,
    postheadspace=1em,
    numberwithin=section,
    headindent=0pt,
    headpunct={.}
    ]{semiimportante}
\declaretheoremstyle[
    spaceabove=\medskipamount,
    spacebelow=\medskipamount,
    headfont=\normalfont\itshape,
    notefont=\normalfont,
    notebraces={(}{)},
    bodyfont=\normalfont,
    postheadspace=1em,
    numberwithin=section,
    headindent=0pt,
    headpunct={.}
    ]{normal}
\declaretheoremstyle[
    spaceabove=\medskipamount,
    spacebelow=\medskipamount,
    headfont=\normalfont\cabeceraespecial,
    notefont=\normalfont,
    notebraces={(}{)},
    bodyfont=\normalfont,
    postheadspace=1em,
    headindent=0pt,
    headpunct={.},
    numbered=no,
    qed=$\square$
    ]{demostracion}

\declaretheorem[name=Remark, style=semiimportante]{remark}
\declaretheorem[name=Corollary, style=semiimportante, numberlike=remark]{corollary}
\declaretheorem[name=Proposition, style=semiimportante, numberlike=remark]{proposition}

\declaretheorem[name=Example, style=normal, numberlike=remark]{example}

\declaretheorem[name=Theorem, style=importante, numberlike=remark]{theorem}
\declaretheorem[name=Definition, style=importante, numberlike=remark]{definition}

\declaretheorem[name=Proof, style=demostracion]{proof}

\def \dis {\displaystyle}

\def \into {\int_\Omega}
\def \confai {-\kern -.5em\rightharpoonup}

\def \ga {\gamma}
\def \Ga {\Gamma}

\def \ep {\varepsilon}
\def \om {\omega}
\def \Om {\Omega}
\def \la {\lambda}

\def \NN {\mathbb N}

\def \RR {\mathbb R}

\def \beq {\begin{equation}}
\def \eeq {\end{equation}}
\def \ba {\begin{array}}
\def \ea {\end{array}}

\def \ecart {\noalign{\medskip}}

\begin{document}

\begin{frontmatter}
	\title{Strong maximum principle for nonlocal diffusion in measure spaces: new characterization and population dynamics applications}
	%Competition model with interface]
	\author[1]{Ana Casado-Sánchez}\ead{acasado2@us.es}
	\author[2]{M\'onica Molina-Becerra} \ead{monica@us.es}
	\author[1]{Antonio Su\'arez\corref{cor}}
	\ead{suarez@us.es}\cortext[cor]{Corresponding author}
	
	\address[1]{Dpto. EDAN and IMUS. Univ. de Sevilla, Avd. Reina Mercedes, s/n, 41013,  Sevilla, Spain} 
	\address[2]{Dpto. Matem\'atica Aplicada II, Escuela Polit\'ecnica Superior and IMUS. Univ. de Sevilla, C/ Virgen de \'Africa, 7, 41011, Sevilla, Spain}

\begin{abstract}
We study the existence of nontrivial solutions to nonlocal semilinear diffusion equations of the form
	\[
	a_0(x)u(x)-\int_\Omega k(x,y)\big(u(y)-u(x)\big)\,d\mu(y)=F(x,u), \quad u\geq 0,
	\]
	on a general measure space $(\Omega,\mathcal{M},\mu)$, where the kernel $k$ is nonnegative and symmetric. In this framework, a major challenge is establishing the strict positivity of solutions. To this end, we provide a novel complete characterization of the strong maximum principle in terms of a connectivity structure induced by the kernel, which leads to a decomposition of the domain into independent components. We also develop the method of sub- and supersolutions in this setting, obtaining existence, uniqueness, and nonexistence results, as well as a regularity result that ensures the continuity of solutions under suitable assumptions. The main interest of working within an abstract measure space lies in its rich applicability, as it unifies and extends a wide variety of frameworks, including discrete, continuous, and hybrid models, or systems interacting across different dimensions. Finally, numerical experiments are presented to illustrate the potential of our theoretical results.
\end{abstract}

	
	%\keywords{Coupled systems, migration models, interchange of
		%flux, spatial heterogeneities}
\end{frontmatter}
\section*{Highlights}
\begin{itemize}
	\item We unify nonlocal diffusion models in a general measure space framework, including continuous, discrete, and hybrid cases as particular examples.
	\item The strong maximum principle is fully characterized by a kernel-induced \mbox{connectivity} partition.
	\item As applications, we study  some nonlocal semilinear equations arising in population dynamics and we perform numerical simulations.
\end{itemize}
\section{Introduction}
In this work, we study semilinear stationary nonlocal problems in which the diffusion operator is given by an integral kernel of the form
\[
\int_\Omega k(x,y)\big(u(y)-u(x)\big)\,d\mu(y),
\]
instead of a second-order elliptic differential operator, typically the Laplacian. More precisely, we analyze equations of the form
\begin{equation}\label{intro_problem}
	a_0(x)u(x)-\int_\Omega k(x,y)\big(u(y)-u(x)\big)\,d\mu(y)=F(x,u), 
	\quad u\geq 0,
\end{equation}
where $(\Omega,\mathcal{M},\mu)$ is a measure space, $k$ is a nonnegative and symmetric kernel, and $F$ is a Carathéodory function.

Problems of this type typically arise in population dynamics (see \cite{hutson}). In this context, $\Omega$ represents an ecosystem and $u(x)$ describes the density of individuals of a given population at the point $x$. Moreover, $\int_\Omega k(x,y)u(y)\,d\mu(y)$ represents the influx of individuals arriving at $x$, while $\int_\Omega k(x,y)u(x)\,d\mu(y)$ accounts for the individuals leaving $x$.

Our problem is stationary, corresponding to the equilibrium states of the associated parabolic evolution equation, which are of particular interest in understanding phenomena such as population persistence or extinction.

To deal with this type of problems and obtain nonnegative solutions, one of the most commonly used tools in the literature is the method of sub- and supersolutions. In the local case, fixed point techniques are also frequently employed; however, due to the lack of compactness of the linear operator, classical results such as Schauder’s fixed point theorem are not directly applicable in the nonlocal setting. Our goal is therefore to establish the method of sub- and supersolutions within our framework.

There is an extensive literature devoted to the nonlocal logistic equation in $\mathbb{R}^N$ with Lebesgue measure, both for convolution-type kernels (\cite{WenSun}, \cite{WenSun2}, \cite{Chinos}, \cite{Rossi}) and for more general kernels (see, for instance, \cite{Coville2}). In this work, we consider a more general nonlinearity within an abstract measure space setting. %This approach allows us to unify within a single framework several previously studied models. 
The primary interest of this approach lies in its capacity to unify very different scenarios under a single formulation. Although well-known techniques, such as the method of sub- and supersolutions, can be naturally adapted to this context, the main value of this generalization is its broad applicability. For instance, the following cases are included:
\begin{itemize}
	\item The classical continuous case in $\mathbb{R}^N$ with Lebesgue measure.
	\item The discrete case, where the kernel reduces to a matrix and the equation corresponds to a finite-dimensional system. In this setting, when the kernel is stochastic, the problem is related to stationary Markov chains with nonlinear dispersion.
	\item Hybrid models combining regions of different dimensions.
\end{itemize}
We also mention that measure spaces endowed with a metric have been considered in \cite{Silvia1}-\cite{Silvia3}. However, in our setting no distance is required. %Nevertheless, the scope of those works is different, since although they study the stationary linear problem, their main focus lies on the asymptotic behavior of evolutionary problems.

A fundamental aspect in the analysis of problems of the form \eqref{intro_problem} is the study of the associated linear operator. In the case of the nonlocal logistic equation, it is well known that the existence of positive solutions is determined by the sign of the principal eigenvalue (see, for instance, \cite{Coville2}-\cite{Rossi2}). Motivated by this fact, we study spectral properties in our general setting and analyze the existence of positive eigenfunctions associated with the principal eigenvalue in Theorems \ref{thespk0}, \ref{th:spectrumRa}, and \ref{thautovautof}.

In addition, in order to apply the method of sub- and supersolutions, maximum principles are required. We show that the weak maximum principle generally holds in our framework (see Theorem \ref{pmax}). However, to guarantee strict positivity of solutions, a strong maximum principle is needed.

One of the main contributions of this work is the characterization of the strong maximum principle in terms of the structure of the kernel $k$. To this end, we introduce a partition of the domain based on a notion of connectivity induced by the kernel, see Definitions \ref{eqrelmu}, \ref{omlN}, and \ref{defCEd1}. This construction allows us to decompose the global problem into independent subproblems on each equivalence class (see Theorem \ref{Thscdes}). These classes provide a notion of connectivity depending on whether individuals can move from one region to another.

In particular, we prove that the strong maximum principle holds if and only if the domain consists of a single equivalence class (see Corollary \ref{ppmafes}). Moreover, we show that this characterization improves standard assumptions in the literature, even in the case of a continuous kernel in $\RR^N$ with the Lebesgue measure (see Proposition \ref{hipgen} and Example~\ref{hipgen2}).

Once these tools are established, we develop the method of sub- and supersolutions in Theorem \ref{Teosubsup} and apply it to obtain existence and uniqueness results under suitable assumptions on the nonlinearity (see Theorem \ref{thexistence}). We also prove a nonexistence and a regularity result, showing that, under additional conditions, solutions are not only in $L^2(\Omega)$ but also continuous (Theorems \ref{thnoexistence} and \ref{thregularity}, respectively).

Finally, we apply our theory to the nonlocal logistic equation, obtaining existence and uniqueness results in each equivalence class. This leads to a rich structure of solutions, where regions with zero solution may coexist with regions where the solution is strictly positive.

To illustrate and validate the theoretical results, we include several numerical experiments in both the discrete and continuous settings. In the discrete case, we analyze the associated matrix problem and its relation to stationary Markov chains with nonlinear dispersion, observing how the structure of equivalence classes is reflected in the decomposition of the system and in the appearance of positive or null solutions depending on the sign of the principal eigenvalue.

In the continuous case, we implement the method of sub- and supersolutions, starting from a supersolution and constructing a monotone decreasing sequence converging to the solution. We also study the influence of different logistic-type nonlinearities, showing how the solution varies when modifying the exponent in models of the form $F(x,u)=\lambda u - \nu u^p$ with $\la\in\RR$, $\nu>0$, $p>1$.

Moreover, we present examples in hybrid domains of different dimensions, \mbox{illustrating} the potential of the proposed framework. In one example, we model a two-dimensional ecosystem separated by a one-dimensional structure, which can be interpreted as a river dividing a natural reserve. In another example, we consider two resource centers (cities) together with the territory connecting them, analyzing how the population distributes depending on the influence of each city and the support of the dispersal kernel.

Finally, we also include the Gompertz model within our framework. Although we have not found references addressing the nonlocal Gompertz equation, this model is widely used in biology to describe phenomena such as tumor growth (see, e.g., \cite{Gompertz1}, \cite{Gompertz2}). We show that it fits into our abstract setting and satisfies the hypotheses required to guarantee existence, uniqueness, and regularity of solutions. This further highlights the applicability of our approach beyond the classical logistic case.

To conclude, we briefly describe the structure of the paper. In Section 2 we introduce the functional framework and the main assumptions. In Section 3 we study the associated spectral problem. We analyze maximum principles, characterize the strong maximum principle via the kernel-induced partition and extend the method of sub- and supersolutions in Section 4. %In Section 5 we develop the method of sub- and supersolutions and obtain existence, uniqueness, and nonexistence results. 
Finally, we present numerical experiments illustrating the theoretical findings in Section 5.
\section{Problem setting}
Let $\Om$ be a set endowed with a $\sigma$-algebra $\cal{M}$ and a nonnegative measure $\mu$ defined on this $\sigma$-algebra. In the measure space $(\Om,\cal{M},\mu)$, we focus on the existence of a solution $u\in L_\mu^2(\Om)$ to the following a semilinear problem with a nonlocal diffusion term 
\beq\label{geneq}a_0(x)u(x)-\into k(x,y)\big(u(y)-u(x)\big)\,d\mu(y)=F(x,u),\quad u\geq0.\eeq
In this expression, $k$ satisfies
\beq\label{hipk}
\left\{\ba{l}
\dis k\in L_{\mu\otimes\mu}^2(\Om\times\Om),\quad k(x,\cdot)\in L^1_\mu(\Om)\ \hbox{ a.e. }x\in\Om, \\\ecart\dis
k(x,y)=k(y,x)\geq0,\ \hbox{ a.e. }x,y\in\Om,
\ea\right.
\eeq
$a_0\in L_\mu^\infty(\Om)$ is nonnegative and $F:\Om\times\RR\to\RR$ is a Carathéodory function, i.e. 
\beq\label{caratheodory}F(x,.)\in C^0(\RR),\ \hbox{a.e. }x\in\Om,\quad F(.,s)\ \hbox{ measurable }\forall\,s\in\RR,\eeq
verifying some further conditions that will be exposed later.

%These kinds of problems typically arise in population dynamics, where
%\begin{itemize}
%	\item $u$ represents the density of individuals in $\Omega$.
%	\item $k(x,y)$ measures the proportion of individuals moving from $y$ to $x$ per unit time.
%\end{itemize}
%As a result, the term
%$$\into k(x,y)\big(u(y)-u(x)\big)\,d\mu(y)$$
%corresponds to the difference between the inflow and outflow of individuals at almost everywhere point $x\in\Om$. 

By defining 
\beq\label{defa}\hat{k}:=\into k(x,y)\,d\mu(y),\quad a:=a_0+\hat k,\eeq
problem (\ref{geneq}) can be rewritten as
\beq\label{geneq2}a(x)u(x)-\into k(x,y)u(y)\,d\mu(y)=F(x,u),\quad u\geq0.\eeq
From now on, we denote by $L\in\mathcal{L}\big(L_\mu^2(\Om);L_\mu^2(\Om)\big)$ the operator
$$Lu:=a(x)u(x)-\into k(x,y)u(y)\,d\mu(y),\quad\forall u\in L_\mu^2(\Om).$$
\begin{remark}\label{rem:identk}
	The inverse of $L$ is not compact. As a consequence, Schauder's fixed point theorem does not apply. However, since $L$ is a particular case of Dirichlet form (see \cite{Fukushima}-\cite{Mosco}), it satisfies the maximum principle. Namely, this is a consequence of the following identity, which comes from the symmetry of $k$:
	\beq\label{identk}\ba{l}
	\dis\int_{\Om\times\Om}\hspace{-2mm}k(x,y)u(x)v(y)d\mu(x)d\mu(y)	=\into \hat kuv\,d\mu(x)\\\ecart\dis\quad\quad
	-\frac{1}{2}	\int_{\Om\times\Om}\hspace{-2mm}k(x,y)\big(u(x)-u(y)\big)\big(v(x)-v(y)\big)d\mu(x)d\mu(y),\quad\forall u,v\in L_\mu^2(\Om). 
	\ea
	\eeq
\end{remark} 
%In the next section, we focus on the linear case of (\ref{geneq2}), where we will be studying some results about spectral values, eigenvalues and eigenfunctions. Afterwards, the strong maximum principle will be studied to ensure the existence of strictly positive solutions. Finally, thanks to all these results we will be able to deal with the non linear case by applying the method of sub and supersolutions.
\section{The spectrum of the nonlocal operator $\boldmath L$}
In this section, we present some properties concerning the spectrum of the operator $L$ defined in the previous section. We start with the simple case where $k$ is the null function. In this case, the following result completely characterizes the spectrum of $L$.
%\begin{proposition}
%	There exists $\ga_0\in\RR$ such that for all $\ga>\ga_0$ it is satisfied:
%	\begin{itemize}
%		\item The operator $(L+\ga)u$ is coercitive, i.e. $\forall\, u\in L_\mu^2(\Om)^M$ it holds
%		\beq\label{hipcoe2} \alpha\into u^2d\mu(x)\leq \int_{\Om} (a+\ga)u^2 \,d\mu(x)-\int_{\Om\times \Om}k(x,y)u(x) u(y)\,d\mu(x)d\mu(y).\eeq
%		\item For every $f\in L_\mu^2(\Om)$, there exists a unique solution of
%		\beq\label{pblin}(L+\ga)u=f\ \hbox{ in }\Om.\eeq
%		%Moreover, if $f\geq0$ a.e. $\Om$, then $u\geq0$ a.e. $\Om$.
%		\item The operator $(L+\ga)^{-1}$ transforms nonnegative functions into nonnegative functions.
%	\end{itemize}
%\end{proposition}
%\begin{proof}
%	We start proving the first statement. Thanks to Remark \ref{rem:identk}, the inequality (\ref{hipcoe2}) is equivalent to
%	$$\alpha\into u^2d\mu(x)\leq \int_{\Om} (a+\ga-\hat k)u^2 \,d\mu(x)+\frac{1}{2}\int_{\Om\times \Om}\hspace{-2mm}k(x,y)\big(u(x)- u(y)\big)^2\,d\mu(x)d\mu(y).$$
%	For this to hold it suffices to take
%	$$\ga_0=\hbox{sup ess}\left\{\hat{k}(x)-a(x)\right\}.$$
%	This concludes the proof of coercitiveness. Thanks to said property, the existence and uniqueness of solution for (\ref{pblin}) just follows from Lax-Milgram's theorem.
%	
%	Finally, in order to prove the third statement, it suffices to prove that if $f$ is nonnegative in (\ref{pblin}), then the unique solution of the equation $u$ is also nonnegative. To show this, it is enough to multiply (\ref{pblin}) by $u^-$ and integrate in $\Om$, taking into account (\ref{identk}).
%\end{proof}
%The above result gives us a weak maximum principle. 
\begin{theorem}\label{thespk0}
	Assume that $k\equiv0$ and let $R_a$ (the range of $a$) denote the set defined as
	\beq\label{defrankA}
	\ba{l}
	R_a:=\Big\{\la\in\RR:\mu\Big(a^{-1}\big((\la-\ep,\la+\ep)\big)\Big)>0,\ \forall\, \ep>0\Big\}\\\ecart\dis\quad \ \ =\Big\{\la\in\RR:\frac{1}{|a-\la|}\notin L_\mu^\infty(\Om)\Big\}.
	\ea
	\eeq
	Then, $\sigma(L)$, the spectrum of $L$,  is equal to $R_a$. Moreover, for every $\la\in R_a$, we have:
	\begin{itemize}
		\item If $\mu(a^{-1}(\{\la\}))>0$, then $\la$ is an eigenvalue of $L$, whose eigenfunctions are the functions $u\in L^2_\mu(\Om)$ such that $u=0$ a.e. in $a^{-1}(\RR\setminus\{\la\})$.
		\item Otherwise, $\la$ belongs to the continuous spectrum of $L$.
	\end{itemize}
\end{theorem}
\begin{proof}
	Since $L$ is a continuous self-adjoint operator, we know that $R_a$ is a compact subset of $\RR$ and that the essential spectrum is empty (see \cite{Yosida}). 
	
	Let us prove that $R_a\subseteq\sigma(L)$.
	We consider $\la\in R_a$. Arguing by contradiction, if $\la$ does not belong to the spectrum of $L$, then $L
	-\la I$ is an isomorphism from $L_\mu^2(\Om)$ to $L_\mu^2(\Om)$ whose inverse is a continuous operator, i.e. there exists $C>0$ such that for all $u\in L_\mu^2(\Om)$ we have
	$$\|u\|_{L_\mu^2(\Om)}\leq C\big\|(L-\la I)u\big\|_{L_\mu^2(\Om)},\quad \forall\, u\in L_\mu^2(\Om).$$
	In particular, taking $u=\chi_{A_\ep}$ with 
	$$ A_\ep\subset a^{-1}\big((\la-\ep,\la+\ep)\big),\quad  \mu(A_\ep)>0,\ \ep>0$$
	then we obtain
	\beq\label{ineqcontra}\|u\|_{L_\mu^2(\Om)}^2=\mu(A_\ep)\leq C^2\int_{A_\ep} |a(x)-\la|^2d\mu(x)\eeq
	As a result,
	$$1\leq 2C^2\ep^2$$
	for all $\ep>0$, which leads to a contradiction, as $\varepsilon\to0$. 
	
	To prove the reverse inclusion, we have to prove that if $\la\in\RR\setminus R_a$, then $\la$ belongs to the resolvent of $L$, but this just follows from 
	$$\frac{1}{|a-\la|}\in L^\infty_\mu(\Om).$$
	To finish the proof, we assume $\la$ an eigenvalue of $L$ and $u\in L^2_\mu(\Om)\setminus\{0\}$ an eigenfunction corresponding to $L-\la I$. Then
	$$(a-\la)u=0\ \hbox{ a.e. }\Om$$
	and therefore
	$$u=0\ \hbox{ a.e. }a^{-1}(\RR\setminus\{\la\}).$$
	Since $u$ is not the null function, $\mu(a^{-1}(\la))>0$. Moreover, this shows that the \mbox{eigenfunctions} are the functions $u$ vanishing in $a^{-1}(\RR\setminus\{\la\})$.
\end{proof}
In the case where $k$ is not the null function, we have
\begin{theorem}\label{th:spectrumRa}
	The set $R_a$ defined in (\ref{defrankA})
	is contained in the spectrum of $L$. The \mbox{remaining} elements of the spectrum are composed by a family of eigenvalues $\la$ at most countable such that the corresponding eigenspace has finite dimension equal to the codimension of the range of $L-\la I$.
\end{theorem}
\begin{proof}
To prove that $R_a\subseteq\sigma(L)$, we argue by contradiction in a similar way than in the proof of Theorem \ref{thespk0} in order to obtain \eqref{ineqcontra}. In this case, however, an additional term appears, and we arrive at
	$$1\leq 2C^2\Big(\ep^2+\int_{A_\ep\times A_\ep}|k(x,y)|^2d\mu(x)d\mu(y)\Big)$$
	Taking $A_\ep$ such that $\mu(A_\ep)\to0$, the right-hand side of the inequality tends to zero when $\ep\to0$, which leads to a contradiction. This shows that $R_a$ is contained in the spectrum of $L$. 
	
	If $\la\in\sigma(L)\setminus R_a$, as in the proof of Theorem \ref{thespk0}, $\exists\ep>0$ such that $|a(x)-\la|\geq \ep$ a.e. $\Om$. This implies that the operator $L-\la I$ 
	consists of a compact perturbation of the operator $u\to (a-\la)u$, which is bijective and continuous with continuous inverse in $L_\mu^2(\Om)$. Thanks to this, Fredholm's alternative theorem applies to obtain that $\la$ is an eigenvalue whose associated eigenspace has finite dimensional equal to the codimension of the range. 
	
	On the other hand, since $L$ is self-adjoint, then it is known (see \cite{Yosida}) the amount of eigenvalues is at most countable.
\end{proof}
From now on, we denote the principal spectral value of $L$ by
\beq\label{defautovppal}\ba{l}
\dis\la_p(L):= \operatorname*{ess\,inf}_{u\not \equiv0}\, \frac{\dis\left <Lu,u\right >_{L^2_\mu(\Om)}}{\|u\|_{L_\mu^2(\Om)}^2}\\\ecart\dis
\qquad\quad=\operatorname*{ess\,inf}_{u\not \equiv0}\, {\dis\into au^2d\mu(x)-\int_{\Om\times\Om} k(x,y)u(y)d\mu(y) \,u(x)d\mu(x)\over \|u\|_{L_\mu^2(\Om)}^2}.\ea\eeq
As a consequence of Theorem \ref{th:spectrumRa}, the following result holds
\begin{theorem}\label{thautovautof}
	The following inequality is satisfied:
		$$\la_p(L)\leq\operatorname*{ess\,inf}\, a.$$
	Moreover, if $\la_p(L)<\operatorname*{ess\,inf}\, a$, then $\la_p(L)$ is an eigenvalue. Furthermore, for every eigenfunction $u\in L^2_\mu(\Om)$ corresponding to $\la_p(\Om)$, we have $|u|$, $u^+$ and $u^-$ are also eigenfunctions associated with $\la_p(\Om)$
\end{theorem}
\begin{proof}
	Thanks to Theorem \ref{th:spectrumRa}, we know that $R_a$ is contained in the spectrum of $L$, which implies that
	$$\la_p(L)=\min\left\{\hbox{sp}(L)\right\}\leq \min\left\{R_a\right\}= \operatorname*{ess\,inf}\left\{a\right\}.$$
	Moreover, if $\la_p(L)<\operatorname*{ess\,inf}\, a$, then $\la_p(L)\notin R_a$. Hence, Theorem \ref{th:spectrumRa} 
	shows that $\la_p(L)$ is an eigenvalue.
	
	We observe that for every $u\in L^2_\mu(\Om)$, we have
	$$\ba{l}
	\dis \into L|u||u|d\mu(x)
	=\into a_0|u|^2 d\mu(x)+\frac{1}{2}\int_{\Om\times\Om}k(x,y)\big(|u(x)|-|u(y)|\big)^2d\mu(x)d\mu(y)\\ \ecart\dis
	\leq \into a_0|u|^2 d\mu(x)+\frac{1}{2}\int_{\Om\times\Om}k(x,y)\big(u(x)-u(y)\big)^2d\mu(x)d\mu(y)
	= \into Luu\,d\mu(x).
	\ea$$
	Then, taking into account that $u$ is an eigenfunction if and only if it attains the minimum in (\ref{defautovppal}), we conclude that $|u|$ is also an eigenfunction, since it yields the same minimum value. 
	
	To conclude, it is enough to use that
	$$u^+=\frac{|u|+u}{2},\quad u^-=\frac{|u|-u}{2},$$
	which implies that $u^+$ and $u^-$ are eigenfunctions if $u$ is an eigenfunction.
\end{proof}

We have finished our study of the spectrum of $L$ with the result above. In the next section, we focus on characterizing the strong maximum principle which will be of the utmost importance to ensure the existence of strictly positive solutions.

\section{A full characterization of the strong maximum principle}
In this section, we prove the existence of a partition of $\Om$ into regions where problem (\ref{geneq2}) can be solved independently and where the strong maximum principle holds. This improves the results obtained by other authors (for example \cite{Rossi2},\cite{Coville}) even in the case where $\Om=\RR^N$, $k$ is continuous and $\mu$ is the Lesbesgue measure, as assumed in these papers. From a modeling perspective, it also enables us to consider different areas of the ecosystem whose connectivity depends on $k$.

First, we recall again that the weak maximum principle holds for the operator $L$. Although the result is well known, we include a short proof for the reader’s convenience:
\begin{theorem}\label{pmax}
	If $\la_p(L)>0$, it holds
	\beq\label{ec:pmax}u\in L^2_\mu(\Om), \quad Lu\geq0\ \hbox{a.e. }\Om\Rightarrow u\geq0\ \hbox{a.e. }\Om.\eeq
\end{theorem}
\begin{proof}
	Let $u$ be in the conditions of the left-hand side of (\ref{ec:pmax}). By (\ref{identk}) and the definition (\ref{defautovppal}) of $\la_p(\Om)$, similarly to the proof of Theorem \ref{thautovautof}, we have
	$$\ba{l}
	\dis0\geq -\into Luu^-d\mu(x)\\\ecart\dis
	=-\into a_0uu^- d\mu(x)-\frac{1}{2}\int_{\Om\times\Om}k(x,y)\big(u(x)-u(y)\big)\big(u^-(x)-u^-(y)\big)d\mu(x)d\mu(y)\\ \ecart\dis
	\geq	\into Lu^-u^-d\mu(x)\geq \lambda_p(L) \into |u^-|^2d\mu(x).
	\ea$$
	This proves that $u^-$ is the null function and therefore $u$ is nonnegative.
\end{proof}
From now on, we assume $k$ satisfies (\ref{hipk}) together with the following assumption: there exists a family of measurable sets $E_i\in\cal{M}$, $i\in I\subseteq\NN$, and a set $J\subseteq\NN\times\NN$ such that it holds
\beq\label{hipkeqgen}
\left\{\ba{l}
 \mu\left(E_i\cap E_j\right)=0\hbox{ if }i\neq j,\quad\forall(i,j)\in\NN\times\NN,\\\ecart\dis
(i,j)\in J \Rightarrow(j,i)\in J,\\\ecart\dis
\forall i\in I \ \exists j\in\NN \hbox{ such that }(i,j)\in J, \\\ecart\dis
\mu\Big(\left\{(x,y)\in\Om\times\Om : k(x,y)>0\right\} \triangle \big(\cup_{(i,j)\in J}E_i\times E_j\big)\Big)=0,
\ea\right.
\eeq
where $\triangle$ denotes the symmetric difference of sets. In this framework, we introduce the following relation:
\begin{definition}\label{eqrelmu}
	Let $i,j\in\NN$. We say that $i$ is related to $j$, and we denote this by $i\stackrel{\cal R}\sim j$, if there exists a sequence of index $t_k\in\NN$, $0\leq k\leq n$ such that $t_0=i$, $t_n=j$ and $(t_k,t_{k+1})\in J$ for all $k\in \{0,...,n-1\}.$
\end{definition} 
We have:
\begin{proposition}\label{proeqrelmu}
	The relation $\mathcal{R}$ given by Definition \ref{eqrelmu} is an equivalence relation on $\NN$.
\end{proposition}

\begin{proof}
	We need to prove that the following properties hold: 
	\begin{itemize}
		\item Reflexivity. Let $i\in\NN$. Then, there exists $j\in\NN$ such that $(i,j)\in J$, which implies $(j,i)\in J$. Then, we can choose the sequence of index $t_0=i$, $t_1=j$, $t_2=i$, showing that $i\stackrel{\cal R}\sim i$.
		\item Symmetry. If $i\stackrel{\cal R}\sim j$ and $t_0,\hdots,t_n$ is a sequence of index connecting $i$ and $j$, it is trivial again by symmetry of $J$ that the reversed sequence $t_n,\hdots,t_0$ connects $j$ and $i$. As a result, $j\stackrel{\cal R}\sim i$.
		\item Transitivity. If $i\stackrel{\cal R}\sim j$ through $t_0,..,t_n$ and $j\stackrel{\cal R}\sim z$ 
		through $q_0,..,q_m$, then $j\stackrel{\cal R}\sim z$ through the concatenated sequence $t_0,..,t_n,q_1,...,q_m$.
	\end{itemize}
	This implies that $\cal{R}$ is an equivalence relation on $\NN$ which concludes the proof.
\end{proof}

Thanks to Proposition \ref{proeqrelmu}, we can define
\begin{definition}\label{omlN}
	For every class $I_l$ of related index, we define
	$$\Om_l=\bigcup_{i\in I_l}E_i,\quad\forall l\geq1$$
	and we denote by $\mathcal{N}$ the following set
	$$\mathcal{N}=\Om\setminus\bigcup_{l\geq1}\Om_l.$$
\end{definition}

	The sets $\Om_l$ with $l\geq1$ form a partition of $\Om\setminus \mathcal{N}$ thanks to (\ref{hipkeqgen}) and Proposition~\ref{proeqrelmu}. We present a simple example where we can contruct a family of measurable sets and pairs of index of the form \eqref{hipkeqgen}:
	
\begin{example}\label{ejmonica}
	Assume $\Om=[0,3]$, $\mu$ the Lebesgue's measure and $k$ defined as
	
	\noindent\begin{minipage}{0.69\textwidth}
		\[
		k(x,y)=\left\{
		\begin{array}{l}
			1,\quad \hbox{if }(x,y)\in[2,3]\times[2,3],\\\ecart\dis
			1,\quad \hbox{if }(x,y)\in[0,1]\times[1,2]\cup[1,2]\times[0,1],\\\ecart\dis
			0,\quad \hbox{in other case}.
		\end{array}
		\right.
		\]
	\end{minipage}
	\hfill
	\begin{minipage}{0.3\textwidth}
		\begin{figure}[H] 
			\centering
			\includegraphics[width=\textwidth]{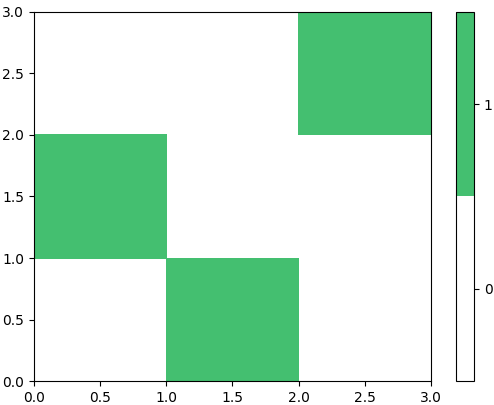}
			\vspace{-7mm} 
			\caption{$k(x,y)$.} 
			\end{figure}
	\end{minipage}
	In this case, we define $E_1=[0,1]$, $E_2=[1,2]$, $E_3=[2,3]$ and $J$ as
	$$J=\{(1,2),(2,1),(3,3)\}.$$
	We observe that these sets satisfy \eqref{hipkeqgen}. Therefore, the equivalence relation given by Proposition \ref{proeqrelmu}, allows us to define the following sets using Definition \ref{omlN}:
	$$\Om_1=[0,2],\quad\Om_2=[2,3],\quad\mathcal{N}=\emptyset.$$
\end{example}
	In Definition \ref{omlN}, each $\Om_l$ represents an isolated region with no interaction with the rest of the domain. This statement will be a consequence of Theorem \ref{Thscdes}, where we show that problems of the form
\beq\label{ecuesc}
au-\into k(x,y)u(y)\,d\mu(y)=f,\quad\hbox{a.e. }\Om,
\eeq
can be solved independently on each $\Om_l$ and in $\mathcal{N}$. 

%A similar result will hold in the semilinear case (\ref{geneq2})** mejor poner un remark:
\begin{theorem}\label{Thscdes}
	We consider $f\in L_\mu^2(\Om)$ and we assume $\la_p(L)>0$. For each $\Om_l$, $l\geq1$, we denote by $u_l$ the solution of the problem 
	\beq\label{pbCE}au_l-\int_{\Om_l} k(x,y)u_l(y)\,d\mu(y)=f,\quad\hbox{a.e. }\Om_l.\eeq
	Then, the solution of problem (\ref{ecuesc}) is given by the function $u$ defined by
	\beq\label{pbCEb} u(x)=\left\{\ba{ll}
	\dis u_l(x) &\hbox{ a.e. }x\in \Om_l,\quad\forall l\geq1\\ \ecart\dis 
	f(x)/a(x)&\hbox{ a.e. }x\in  \mathcal{N}.\ea\right.\eeq
\end{theorem}

\begin{proof}
	Let us consider $u$ given by (\ref{pbCEb}), which is well defined thanks to Lax-Milgram's theorem. We show that $u$ is the solution of (\ref{ecuesc}).
	
	Fix a point $x\in \Om_{l_1} $, where $l_1\geq1$. If $y$ is a point belonging to $\Om_{l_2}\neq\Om_{l_1}$, thanks to (\ref{hipkeqgen}), we know that $k(x,y)=0$; otherwise, $\Om_{l_1}=\Om_{l_2}$ a.e. Moreover, $k(x,y)=0$ for all $y\in \mathcal{N}$. This proves that for almost every $x\in \Om_{l_1}$, we have
	$$\ba{c}\dis a(x)u(x)-\int_\Om k(x,y)u(y)\,d\mu(y) =a(x)u(x)-\int_{\Om_{l_1}} k(x,y)u(y)\,d\mu(y)\\ \ecart\dis=a(x)u_{l_1}(x)-\int_{\Om_{l_1}} k(x,y)u_{l_1}(y)\,d\mu(y)=f(x).\ea$$
	Furthermore, if $x\in \mathcal{N}$, then
	$$a(x)u(x)-\into k(x,y)u(y)\,d\mu(y)=a(x){f(x)\over a(x)}=f(x),$$
	where we recall that thanks to Theorem \ref{thautovautof}, $a\geq\la_p(L)>0$.
	This proves that $u$ is indeed the solution of (\ref{ecuesc}). 
\end{proof}
\begin{remark}
	In the nonlinear case a similar result holds if the nonlinear problem in each domain has a solution.
\end{remark}
Thanks to the results proved in this section, assuming $L$ in the conditions of Theorem~\ref{pmax}, we prove that the strong maximum principle holds in every equivalence class and, as consequence, we characterize when it holds in the whole $\Om$.

\begin{theorem} \label{ppmxc} Let $l\geq1$ such that
	$\la_p(L_{\Om_l})>0$, where $\la_p(L_{\Om_l})$ is the first spectral value of the operator $L$ restricted to $\Om_l$. Then, if 
	$f\in L_\mu^2(\Om_l)$ is nonnegative and not the null function, the solution $u_l$ of (\ref{pbCE}) satisfies 
	\beq\label{ppomaCE} u_l(x)>0,\ \hbox{ a.e. } x\in \Om_l.\eeq
\end{theorem}
\begin{proof} Taking into account that 
	\beq\label{crecu}u_l(x)={1\over a(x)}\Big(f(x)+\int_{\Om_l} k(x,y)u_l(y)\,d\mu(y)\Big),\quad \ \hbox{ a.e. } x\in \Om_l,\eeq
	and that by Theorem \ref{pmax} (applied to $\Om$ replaced by $\Om_l$) $u_l\geq 0$ in $\Om_l$, we have 
	$u_l\geq f/a$ and then $u_l>0 $  a.e. in $\{f>0\}$.\par
	We denote
	$$A_0=\big\{x\in \Om_l:\ f>0\}$$
	Then, assuming we have defined the set $A_i$, $i\geq 0$, we take
	$$A_{i+1}=\Big\{x\in \Om_l:\ \int_{A_i}k(x,y)d\mu(y)>0\Big\}\cup A_i,$$
	Thanks to Definitions \ref{eqrelmu} and \ref{omlN}, we have that $\{A_i\}$ is an increasing family of subsets of $\Om_l$ which fulfills $\Om_l$. Now, we observe that  $u_l>0$ a.e. in $A_0$, equality (\ref{crecu}) implies  $u_l>0$ a.e. in $A_1$. Iterating this reasoning, we deduce $u_l>0$ a.e. in every set $A_i$ and then the result.
\end{proof}

As a result of this theorem, the following corollary holds, which characterizes the strong maximum principle.
\begin{corollary} \label{ppmafes} 
	The strong maximum principle is satisfied if and only if $\mu(\mathcal{N})=0$ and the partition given in Definition \ref{omlN} is formed by a single element $\Om_1$ with $\la_p(\Om_1)>0$.
\end{corollary}
\begin{proof}
	We assume that the strong maximum principle holds. Arguing by contradiction, we assume that $\mu(\mathcal{N})>0$. Then, by Definition \ref{omlN} and (\ref{hipkeqgen}), $k(x,y)=0$ a.e. $x\in\mathcal{N}$, a.e. $y\in \Om$. In particular, let $f$ be nonnegative such that $f(x)=0$ a.e. $x\in\mathcal{N}$. We obtain
	$$0=f(x)=a(x)u(x)-\into k(x,y)u(y)d\mu(y)=a(x)u(x).$$
	Thus, $u(x)=0$, which contradicts our hypothesis.
\end{proof}
\begin{remark}
	We note that constructing a family of sets $E_i$ satisfying \eqref{hipkeqgen} may be difficult when the domain cannot be decomposed as a finite union of cubes.
	
	In such situations, even in simple domains such as a ball, a finite union of cubes can only approximate the domain: for any $\varepsilon>0$, one may construct a set $\Omega_\varepsilon$ such that 
	$$\mu(\Omega\setminus\Omega_\varepsilon)<\varepsilon.$$
	
	Nevertheless, it is important to emphasize that the validity of the results in the previous sections does not rely on having an explicit geometric description of the sets $E_i$, but only on the existence of such a decomposition.
	
	However, when $k$ is continuous and $\mu$ is a Radon measure, we provide an alternative construction which is simpler, applies to more general domains and also satisties our theoretical results. Therefore, it can be used in practical applications.
\end{remark}

We now consider the case where $\Omega$ is a locally compact Hausdorff space, $\mu$ is a Radon measure, and $k\in C^0(\Omega\times\Omega)$. This setting includes the classical case appearing in the literature (see \cite{Coville3}), where $\Omega$ is an open subset of $\mathbb{R}^N$ and $\mu$ is the Lebesgue measure. To the best of our knowledge, the results obtained in this section are new even in this framework.

In this setting, in order to apply Theorems~\ref{Thscdes} and \ref{ppmxc}, we introduce a simpler construction of a partition of $\Omega$ than that given in Definition~\ref{omlN}, which also applies to more general domains.
\begin{definition}
	\label{defCEd1} Let $x,y\in\Om$, we say that $x$ is related to $y$ (through $k$)and we denote this by $x\stackrel{\cal R'}\sim y$, if there exists a sequence of points $p_i\in\Om$, $0\leq i\leq n$ such that $p_0=x$, $p_n=y$ and $k(p_j,p_{j+1})>0$ for all $j\in \{0,...,n-1\}.$
	\par
	We define the following set
	$${\cal N'}=\Big\{ x\in \Om:\ k(x,y)=0,\quad\forall\,y\in\Om\Big\}.$$
\end{definition}
In this case, the equivalence classes are subsets $\Om_l\subset\Om$ instead of index. Observe that in this case these sets are open and $\mathcal{N}$ is closed in $\Om$.

We also emphasize that in this case the following regularity result is satisfied
\begin{proposition}
	We assume $\la_p(\Om)>0$. If $f\in C^0(\Om)$ and $a\in C^0(\Om)$ in (\ref{ecuesc}), then $u\in C^0(\Om)$
\end{proposition}
\begin{proof}
	This just follows from
	$$u=\frac{1}{a}\left(f+\into k(x,y)\,u(y) d\mu(y)\right),$$
	where the right-hand side is continuous.
\end{proof}

In \cite{Coville3}, in the case where $\Omega\subset\mathbb{R}^N$ is open and $\mu$ is the Lebesgue measure, the kernel $k$ is assumed to satisfy
\begin{equation}\label{hipcoville}
	\exists c_0>0,\ \exists \varepsilon_0>0,\quad
	\inf_{x\in\Omega}\left\{\inf_{y\in B(x,\varepsilon_0)} k(x,y)\right\}>c_0,
\end{equation}
which, thanks to the continuity of $k$ on $\Omega\times\Omega$, implies in particular that $k(x,x)>0$ in $\Omega$.

In the present work, we show that this assumption is not intrinsic to the validity of the strong maximum principle. More precisely, we provide a necessary and sufficient condition for the strong maximum principle to hold in Corollary \ref{ppmafes}. We note that assumptions such as \eqref{hipcoville} may naturally appear in works pursuing additional qualitative properties beyond the strong maximum principle. 

In particular, we prove 
%In particular, we prove that \eqref{hipcoville} implies the existence of a single equivalence class in $\Omega$ togeth
%er with $\mathcal N=\emptyset$, and therefore ensures the applicability of the strong maximum principle. 
\begin{proposition}\label{hipgen}
	Assumption \eqref{hipcoville} implies the existence of a single equivalence class in $\Omega$ together with $\mathcal N'=\emptyset$, and therefore, the strong maximum principle applies. 
\end{proposition}

\begin{proof}
	We assume \eqref{hipcoville} and let $x,y\in\Om$. We need to prove that $x\stackrel{\cal R'}\sim y$. Taking into account that $\Om$ is an open and connected subset of $\RR^N$, we have that $\Om$ is path-connected. Then, there exists $\gamma\in C^0([0,1],\Om)$ such that $\ga(0)=x$, $\ga(1)=y$. We denote by
	$$\Ga=\left\{\big(\ga(t),\ga(t)\big): t\in[0,1]\right\}.$$
	Since $\Ga\subset\Om\times\Om$ is compact and $k$ is continuous in $\Om\times\Om$ (and hence uniformly continuous in compact subsets of $\Om\times\Om$) and positive in $\Ga$, we have that $\exists\rho>0$, $\rho<$distance$(\Ga,\partial(\Om\times\Om))$ such that defining by
	$$K=\left\{(x,y)\in\Om\times\Om:\hbox{ distance}\big((x,y),\Ga\big)<\rho\right\},$$
	it is satisfied that
	$$k(x,y)>0,\quad\forall(x,y)\in K.$$
	Moreover, using that $\ga$ is uniformly continuous, we have
	$$\exists n\in\NN \hbox{ such that distance}(\ga(t_1),\ga(t_2))\leq\rho,\quad\forall t_1,t_2\in[0,1],\hbox{ with } |t_1-t_2|\leq\frac{1}{n}.$$
	As a result, $k\big(\ga(t_1),\ga(t_2)\big)>0$, $\forall t_1,t_2\in[0,1],\hbox{ with } |t_1-t_2|\leq1/n.$ For this reason, $x\stackrel{\cal R'}\sim y$ through the sequence
	$$p_0=\ga(0)=x,\, p_1=\ga\left(\frac{1}{n}\right)\hspace{-1.2mm},\hdots, p_k=\ga\left(\frac{k}{n}\right)\hspace{-1.2mm},\hdots,p_{n-1}=\ga\left(\frac{n-1}{n}\right)\hspace{-1.2mm},\, p_n=\ga(1)=y.$$
\end{proof} 

Furthermore, we exhibit a simple example for which \eqref{hipcoville} fails, while $\mathcal N'=\emptyset$ still holds, and hence the strong maximum principle remains valid.

\begin{example}\label{hipgen2}
We present the following example of a function $k$ such that $\mathcal N'=\emptyset$ while $k(x,x)>0$ does not hold in $\Omega$. Let $\Omega=(0,+\infty)$. We first define $f:(0,+\infty)\times(0,+\infty)\to\mathbb R$ by
\[
f(x,y):=
\begin{cases}
	0, & x>1,\\
	1-x, & x\in(0,1].
\end{cases}
\]
We then define $k:(0,+\infty)\times(0,+\infty)\to\mathbb R$ by symmetrization,
\[
k(x,y):=f(x,y)+f(y,x).
\]
Clearly, for all $x\in[1,+\infty)$ we have $k(x,x)=0$. 

Nevertheless, for every $x\in(0,+\infty)$ there exists $y\in(0,+\infty)$ such that $k(x,y)>0$, since it suffices to take any $y\in(0,1)$ (and conversely by symmetry). As a consequence $\mathcal{N'}=\emptyset$ and any two points $x,z\in(0,+\infty)$ satisfy $x\stackrel{\cal R'}\sim z$, since $x\stackrel{\cal R}\sim y$ and $y\stackrel{\cal R'}\sim z$ for all $y\in(0,1)$.

To visually clarify the example, we represent the functions $f$ and $k$ in $\Om\times\Om$.
		\begin{figure}[H]
				\centering
				\includegraphics[width=0.48\textwidth]{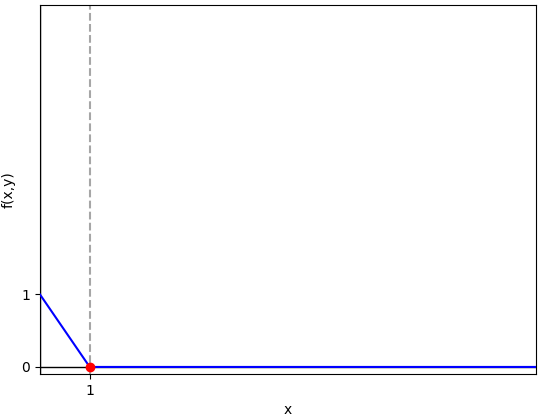}
				\hfill
				\includegraphics[width=0.49\textwidth]{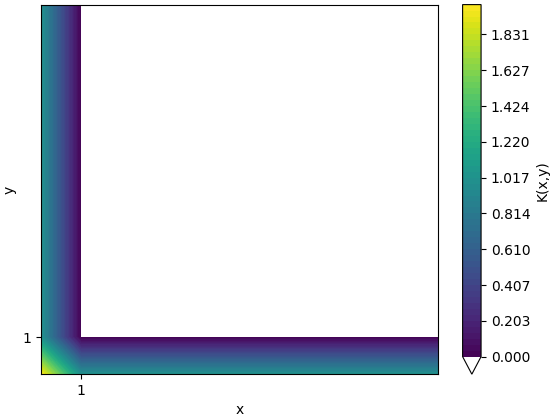}
				\caption{Functions $f(x,y)$ and $k(x,y)$.}
			\end{figure}
\end{example}
%	lo que por continuidad de $K$ implica $K(x,x)>0$ en $\Om$. Sin embargo, nosotros no necesitamos esta hipótesis. De hecho, hemos probado que es condición necesaria y suficiente para tener principio del máximo fuerte $\mathcal{N}_\mathcal{R}=\emptyset$. Cuando no se verifique esta propiedad, gracias al Teorema \ref{Thscdes}, basta con estudiar el problema por separado en $\mathcal{N}_\mathcal{R}$ y $\Om\setminus\mathcal{N}_\mathcal{R}$. \medskip
%	
%	Para ver que, efectivamente, nuestra hipótesis es más débil, 
%\end{remark}

To conclude this section, we show that thanks to Theorem \ref{pmax}, the sub and supersolutions method, can be extended as follows
\begin{theorem}[The method of sub and supersolutions]
	\label{Teosubsup} We consider $\underline u,\overline u\in L^2_\mu(\Om)$ and $F$ verifying (\ref{caratheodory}) such that
	\begin{enumerate}[labelsep=0.5em,label=(\roman*)]
		\item \begin{minipage}[t]{\linewidth}
			\vspace{-2.5ex}
			\beq\label{hipsumF}|F(x,\underline u)|+|F(x,\overline u)|\in L_\mu^2(\Om),\eeq
		\end{minipage}
		
		\item \begin{minipage}[t]{\linewidth}
			\vspace{-2.5ex}
			\beq\label{submsup}\underline u\leq \overline u\ \text{ a.e. in }\Om,\eeq
		\end{minipage}
		
		\item \begin{minipage}[t]{\linewidth}
			\vspace{-2.5ex}
			\beq\label{subysyp} (L\underline u)(x)\leq F(x,\underline u),\quad  (L\overline u)(x)\geq F(x,\overline u),\quad \text{a.e. } x\in \Om,\eeq
		\end{minipage}
		
		\item There exists $\ga\geq0$ such that
		\vspace{-1ex}
			\beq\label{hipcrF} F(x,s_1)+\ga s_1\leq F(x,s_2)+\ga s_2,\quad\underline u(x)\leq s_1\leq s_2\leq \overline u(x),\quad \text{a.e. }x\in\Om.\eeq		
	\end{enumerate}
	Then, there exists $u$ solution of
	\beq\label{ecunol} Lu(x)= F(x, u),\quad\underline u(x)\leq u(x)\leq \overline u(x),\quad \hbox{a.e. }x\in\Om\eeq
	Furthermore, there exist $u_\ast,u^\ast$ such that every solution $u$ of (\ref{ecunol}) verifies 
	$$u_\ast(x)\leq u(x)\leq u^\ast(x), \quad a.e.\hbox{ } x\in \Om.$$
\end{theorem}
\begin{proof} 
As the proof is fairly standard, we’ll just outline it briefly. First, observe that it suffices to prove the result for $\gamma=0$ and under the assumption that $\lambda_p(L)>0$. 
	The general case then follows by applying the first result to $L+\mu I$ and $F(x,s)+\mu s$, 
	with $\mu\geq\gamma$ and $\mu>-\lambda_p(L)$.
	
	We define $\underline u_0=\underline u$, $\overline u_0=\overline u$. For each $n\geq 0$ we denote by   $\underline u_{n+1}$, $\overline u_{n+1}$ the solutions of the equations
	\beq\label{edpn} L\underline u_{n+1}=F(x,\underline u_n),\quad L\overline u_{n+1}=F(x,\overline u_n),\quad \hbox{ a.e. }\Om.\eeq
	These problems have a unique solution thanks to $\la_p(L)>0$, which allows to apply Lax-Milgram's Theorem.\par
	It is direct to prove that  for each $n\geq 0$, it is satisfied
	$$ L\underline u_n \leq F(x,\underline u_n)\ \hbox{ a.e. }\Om,\quad L \overline u_n\geq F(x,\overline u_n)\ \hbox{ a.e. }\Om,$$
	$$\underline u_n\leq \underline u_{n+1}\leq \overline u_{n+1}\leq \overline u_n\ \hbox{ a.e. } \Om.$$
	%Reasoning by induction, it is enough to show that if
	%$$ L\underline u_n \leq F(x,\underline u_n),\quad L \overline u_n\geq F(x,\overline u_n),\quad\underline u_n\leq \overline u_n\ \hbox{ a.e. }\Om,$$
	%then
	%$$ L\underline u_{n+1} \leq F(x,\underline u_{n+1}),\quad L \overline u_{n+1}\geq F(x,\overline u_{n+1})\ \hbox{ a.e. }\Om,$$
	%$$\underline u_n\leq \underline u_{n+1}\leq \overline u_{n+1}\leq \overline u_n\ \hbox{ a.e. }\Om.$$
	%In fact, it suffices to prove the last inequality. Indeed, thanks to (\ref{hipcrF}) said inequality implies that
	%$$L\underline u_{n+1}=F(x,\underline u_n)\leq F(x,\underline u_{n+1})\leq F(x,\overline u_{n+1})\leq F(x,\overline u_n)=L\overline u_{n+1}\ \hbox{ a.e. }\Om.$$\par
	%By definition of $\underline u_{n+1}$ and $\overline u_{n+1}$, and applying hypothesis (\ref{subysyp}) we have
	%$$L\underline u_{n+1}= F(x,\underline u_n)\geq L\underline u_n,\quad L\overline u_{n+1}= F(x,\overline u_n)\leq L\overline u_n,\ \hbox{ a.e. }\Om,$$
	%therefore
	%$$\underline u_{n+1}\geq \underline u_n,\quad \overline u_{n+1}\leq \overline u_n,\quad\hbox{ a.e. }\Om.$$
	%Moreover, by using
	%$$L\underline u_{n+1}=F(x,\underline u_n)\leq F(x,\overline u_n)=L\overline u_{n+1}\ \hbox{ a.e. }\Om,$$
	%we conclude that $\underline u_{n+1}\leq \overline u_{n+1}$ a.e. $\Om$, which finally proves the inequality.
	
	Since the sequences $\left\{\underline{u_n}\right\}_n$ and $\left\{\overline{u_n}\right\}_n$ are monotone and bounded, there exist $u_\ast,u^\ast:\Om\to \RR$ such that
	$$\underline u_n\nearrow u_\ast,\quad \overline u_n\searrow u^\ast \hbox{ a.e. }\Om,\quad \underline u_n\leq u_\ast\leq u^\ast\leq \overline u_n,\ \forall\, n\in\NN.$$
	Using  also
	 $$\underline{u}\leq\underline{u}_n\leq\overline u_n\leq\overline{u},\quad F(x,\underline{u})\leq F(x,\underline{u}_n)\leq F(x,\overline u_n) \leq F(x,\overline{u}),\quad\hbox{a.e. in }\Om,$$ 
	 and (\ref{hipsumF}), we can apply
	Lebesgue's dominated convergence theorem to deduce
	$$\underline u_n\to u_\ast,\ \ \overline u_n\to u^\ast,\ \ F(x,\underline u_n)\to F(x,u_\ast),\ \ F(x,\overline u_n)\to F(x,u^\ast)\quad\hbox{ in }L^2_\mu(\Om).$$
	This allows us to pass to the limit in (\ref{edpn}) to conclude that $u_\ast$ and $u^\ast$ are solutions of  (\ref{ecunol}).
	
	To end the proof, it remains to show that $u_*$, $u^*$ are, respectively, minimal and maximal solutions among all solution $u$ of (\ref{ecunol}) satisfying $\underline u\leq u\leq \overline u$. 
	Again by induction, we can prove that 
%	To this aim, for such $u$, let us prove by induction that
	$$\underline u_n\leq u\leq \overline u_n\ \hbox{ a.e. }\Om,$$
	which  implies that  $u_\ast\leq u\leq u^\ast$.
	
	%Indeed, the inequality is trivial for $n=0$. Assuming it holds for some $n$, we obtain
	%$$L\underline u_{n+1}=F(x,\underline u_n)\leq F(x,u)=Lu\leq F(x,\overline u_n)=L\overline u_{n+1}\ \hbox{ a.e. }\Om.$$
	%Thus, $\underline u_{n+1}\leq u\leq \overline u_{n+1}$ a.e. $\Om$, which concludes the proof. 
\end{proof}

\section{Applications}
In this section, we assume that conditions (\ref{hipk}) and (\ref{caratheodory}) hold. Under further assumptions, we give an existence result to solve some problems of the form (\ref{geneq2}) finding other solutions than the null function. Then, we will develop a nonexistence result. Finally, we present the logistic and Gompertz equations as models from population dynamics fitting into our abstract framework. Moreover, in the particular case of the logistic equation, we will also be able to find strictly positive solutions in each $\Om_l$ given by Definition \ref{omlN}. This extends uniqueness to each equivalence class so that there are multiple nonnegative solutions different from the null function.

First of all, we recall that it is satisfied
$$\Om=\mathcal{N}\bigcup_{l\geq1}\Om_l,$$
where $\Om_l$, $l\geq1$ and $\mathcal{N}$ are given by Definition \ref{omlN}.
From now on, without loss of generality, we can assume 
\beq\label{hipstromaxprin}\exists l\geq1 \hbox{ such that }\Om=\Om_l\ \hbox{ a.e in }\Om.\eeq
Otherwise, the problem can be studied independently in each $\Om_l$ and $\mathcal{N}$ similarly to Theorem \ref{Thscdes}. 

We begin showing the following existence theorem:
\begin{theorem}\label{thexistence}
	Let $\Om$ be such that $\mu(\Om)<+\infty$ and $k\in L^\infty_{\mu\otimes\mu}(\Om\times\Om)$. We consider $F:\Om\times[0,+\infty)\to\RR$ satisfying (\ref{caratheodory}) together with the following assumptions:
		\begin{enumerate}[labelsep=0.5em,label=(\roman*)]
		\item \begin{minipage}[t]{\linewidth}
			\vspace{-2.5ex}
			\beq\label{Fnoneg0}F(x,0)\geq0.\eeq
		\end{minipage}
		
		\item There exists $\alpha\in L^\infty_\mu(\Om)$ satisfying that $\la_p(L-\alpha I)<0$ and for every $\ep>0$ $\exists\delta(\ep)>0$ such that  if $0<s\leq\delta$, then
		\beq\label{hipsubsol}F(x,s)\geq(\alpha(x)-\ep)s\ \hbox{ a.e. }x\in\Om.\eeq
		
		\item For every $M>0$ there exists $R(M)>0$ such that if $s\geq R$, then
		\beq\label{hipsupsol}F(x,s)<-Ms \ \hbox{ a.e. }x\in\Om.\eeq
	\end{enumerate}
	Then, there exist a strictly positive solution of \eqref{geneq2} $u\in L^2_\mu(\Om)$ and minimal and maximal strictly positive solutions of \eqref{geneq2} $u_\star,u^\star\in L^2_\mu(\Om)$ such that $u\in L^2_\mu(\Om)$  satisfies $u_\star\leq u\leq u^\star$. 
	
	Furthermore, if $F$ satisfies
	\begin{equation}\label{hipunicidad}
		\frac{F(x,s_1)}{s_1}>\frac{F(x,s_2)}{s_2},\quad \forall s_1, s_2, \ 0<s_1< s_2<+\infty
	\end{equation}
	then, problem \eqref{geneq2} admits a unique strictly positive solution.
\end{theorem}
\begin{proof}
	The idea is to apply the method of sub and supersolutions established in Theorem~\ref{Teosubsup}, since we seek a nonnegative and nontrivial solution. By hypothesis, we know that $\la_p(L-\alpha I)<0$. Taking Theorem \ref{thautovautof} into account, we distinguish two cases:
	\begin{itemize}
		\item Case $\la_p(L-\alpha I)<\operatorname*{ess\,inf}\{a-\alpha \}$. By Theorem \ref{thautovautof}, there exists $\varphi\geq0$, $\varphi\not\equiv0$ an eigenfunction in $L_\mu^\infty(\Om)$ thanks to
		 $k\in L^\infty_{\mu\otimes\mu}(\Om\times\Om)$. Let $\tau>0$, we define
		$$\underline{u}=\tau\varphi.$$
		For $\underline{u}$ to be a subsolution, it must be satisfied
		$$L\underline{u}\leq F(x,\underline u)\Leftrightarrow (L-\alpha)\underline{u}\leq F(x,\underline u)-\alpha\underline{u}\Leftrightarrow\tau\varphi\la_p(L-\alpha I)\leq F(x,\tau \varphi)-\alpha\tau\varphi,\ \hbox{a.e. }\Om.$$
		Let us consider $\ep>0$. If we choose $\tau$ such that $\tau\|\varphi\|_{L^\infty_\mu(\Om)}\leq\delta(\ep)$, then condition (\ref{hipsubsol}) implies
		$$F(x,\tau\varphi)\geq(\alpha-\ep)\tau\varphi\ \hbox{a.e. }\Om.$$
		It thus suffices to verify
		$$\tau\varphi\la_p(L-\alpha I)\leq\tau\varphi(\alpha-\ep-\alpha)\Leftrightarrow\varphi\la_p(L-\alpha I)\leq-\ep\varphi,\ \hbox{a.e. }\Om.$$
		This inequality trivially holds wherever $\varphi(x)=0$. On the other hand, if $\varphi(x)>0$, it is reduced to 
		$$\la_p(L-\alpha I)\leq-\ep,$$
		which is satisfied due to the assumption $\la_p(L-\alpha I)<0$ by taking $\ep$ small enough.
		\item Case $\la_p(L-\alpha I)=\operatorname*{ess\,inf}\{a-\alpha \}<0$. If we consider $\ep>0$, then (\ref{hipsubsol}) implies that there exists $\delta(\ep)>0$ such that $F(x,\delta)\geq(\alpha-\ep)\delta$. We denote by $A_\ep$ the set
		\beq\label{Ar}A_\ep=\{x\in\Om\, |\, a(x)-\alpha(x)<\la_p(L-\alpha I)+\ep\}.\eeq
		We observe that $\mu(A_\ep)>0$, hence we can define the following nonnegative and nontrivial function
		$$\underline{u}=\delta\chi_{A_\ep}.$$
		Similarly to the previous case, it must be satisfied
		$$(L-\alpha)\delta\chi_{A_\ep}\leq F(x,\delta \chi_{A_\ep})-\alpha\delta\chi_{A_\ep},\ \hbox{a.e. in }\Om,$$
		equivalently,
		$$(a-\alpha)\delta\chi_{A_\ep}-\delta\int_{A_\ep}k(x,y)\,dy\leq
		F(x,\delta \chi_{A_\ep})-\alpha\delta\chi_{A_\ep}.$$
		Taking into account $-\delta\int_{A_\ep}k(x,y)\,dy\leq0$, it suffices to show that
		$$(a-\alpha)\delta\chi_{A_\ep}\leq F(x,\delta \chi_{A_\ep})-\alpha\delta\chi_{A_\ep}.$$
		Almost every $x\in\Om$ such that $\chi_{A_\ep}(x)=0$ the inequality holds thanks to \eqref{Fnoneg0}. Furthermore, almost every $x\in\Om$ such that $\chi_{A_\ep}(x)=1$, by applying again \eqref{hipsubsol}, it is enough to prove
		$$a-\alpha\leq-\ep,\ \hbox{a.e. }A_\ep.$$
		Condition \eqref{Ar} implies
		$$a-\alpha<\la_p(L-\alpha I)+\ep,\ \hbox{a.e. }A_\ep,$$
		hence it suffices to show that
		$$2\ep+\la_p(L-\alpha I)\leq0.$$
			which is satisfied due to the assumption $\la_p(L-\alpha I)<0$ by taking $0<\ep\leq-\la_p(L-\alpha I)/2$.		
	\end{itemize}
	We have found nonnegative and nontrivial subsolutions in both cases. We look for a supersolution next. For this, we consider $M>0$, thus there exists $R(M)>0$ such that (\ref{hipsupsol}) holds. For $R(M)$ to be a supersolution, we need the following to be satisfied:
	$$L\overline u\geq F(x,\overline u)\Leftrightarrow aR-R\hat k\geq F(x,R),\ \hbox{ a.e. }x\in\Om.$$
	Since $F(x,R)<-MR$, it is enough to prove
	$$aR-R\hat k\geq-MR\Leftrightarrow M\geq\hat k-a,\ \hbox{a.e. in }\Om.$$
	As a result, we choose
	$$M\geq\max\left\{\|\hat k-a\|_{L^\infty_\mu(\Om)},\tau\|\varphi\|_{L^\infty_\mu(\Om)},-\frac{\la_p(L-\alpha I)}{2}\right\}\quad\overline u=R(M).$$
	Once $\underline u\leq\overline u$ have been found, we have that assumptions \eqref{hipsumF}, \eqref{submsup} and \eqref{subysyp} of Theorem \ref{Teosubsup} hold. To verify that \eqref{hipcrF} holds, we take
	\beq\label{defgammamon}\ga\geq\frac{F(x,s_1)-F(x,s_2)}{s_2-s_1},\quad\forall s_1,s_2\in[0,M],\quad s_1< s_2,\quad \hbox{a.e. }x\in\Om.\eeq
	Consequently, thanks to Theorem \ref{Teosubsup}, we conclude that there exists $u$ solution of \eqref{geneq2} with $\underline{u}\leq u\leq\overline{u}$ a.e. in $\Om$ and, hence, $u$ is nonnegative and nontrivial.
	
	Moreover, we prove that the nontrivial solution obtained is strictly positive. We remember that in the method of sub and supersolutions we build two monotone sequences $\{\underline{u}_n\}$ and $\{\overline{u_n}\}$ that converge to $u_\star$ (minimal solution) and $u^\star$ (maximal solution), respectively. In particular, we recall that $\underline{u}_0=\underline{u}$ and $\underline{u}_1$ is the solultion of
	$$L\underline{u}_{1}=F(x,\underline{u})$$
	On the other hand, we know that $F$ is monotone. By choosing $s_1=0$ and $s_2=\underline u$ in \eqref{defgammamon} and applying \eqref{Fnoneg0}, we have
	$$F(x,\underline{u})+\ga\underline{u}\geq F(x,0)\geq0,$$
	which implies
	$$L\underline{u}_1+\ga\underline{u}=F(x,\underline{u})+\ga\underline{u}\geq0.$$
	Hence, thanks to \eqref{hipstromaxprin} and Corollary \ref{ppmafes}, we can apply the strong maximum principle to conclude $\underline{u}_1>0$ a.e. in $\Om$. Since $u\geq\underline{u}_1$, the solution is strictly positive.
	
	To finish the proof, we assume \eqref{hipunicidad} and we consider $u_1$, $u_2$ strictly positive solutions of \eqref{geneq2} with $u_1\leq u_2$. We define the operators $L_{u_i}:L^2_\mu(\Om)\to L^2_\mu(\Om)$, $i=1,2$ as follows
	$$L_{u_i} v:=\left(a-\frac{F(x,u_i)}{u_i}\right)v-\into k(x,y)v(y)\, d\mu(y),\quad \forall v\in L^2_\mu(\Om),\ i=1,2.$$
	Since $u_1$ is solution of \eqref{geneq2}, it is a positive eigenfunction associated with the zero eigenvalue  of $L_{u_1}$. Thanks to the strong maximum principle, we conclude that $\la_p(L_{u_1})=0$. Analogously, we obtain $\la_p(L_{u_2})=0$. By using the definition of principal spectral value \eqref{defautovppal} and by applying \eqref{hipunicidad}, we have	
	$$0={\dis\into \left(a-\frac{F(x,u_1)}{u_1}\right)u_1^2(x)\,d\mu(x)-\int_{\Om\times\Om} k(x,y)u_1(y)d\mu(y) \,u_1(x)d\mu(x)\over \|u_1\|_{L_\mu^2(\Om)}^2}$$
	$$=\min_{v\not \equiv0}{\dis\into \left(a-\frac{F(x,u_1)}{u_1}\right)v^2(x)\,d\mu(x)-\int_{\Om\times\Om} k(x,y)v(y)d\mu(y) \,v(x)d\mu(x)\over \|v\|_{L_\mu^2(\Om)}^2}$$
	$$\leq{\dis\into \left(a-\frac{F(x,u_1)}{u_1}\right)u_2^2(x)\,d\mu(x)-\int_{\Om\times\Om} k(x,y)u_2(y)d\mu(y) \,u_2(x)d\mu(x)\over \|u_2\|_{L_\mu^2(\Om)}^2}$$
	$$<{\dis\into \left(a-\frac{F(x,u_2)}{u_2}\right)u_2^2(x)\,d\mu(x)-\int_{\Om\times\Om} k(x,y)u_2(y)d\mu(y) \,u_2(x)d\mu(x)\over \|u_2\|_{L_\mu^2(\Om)}^2}$$
	$$=\min_{v\not \equiv0}{\dis\into \left(a-\frac{F(x,u_2)}{u_2}\right)v^2(x)\,d\mu(x)-\int_{\Om\times\Om} k(x,y)v(y)d\mu(y) \,v(x)d\mu(x)\over \|v\|_{L_\mu^2(\Om)}^2}=0.$$
	This implies in particular
	$${\dis\into \left(a-\frac{F(x,u_1)}{u_1}\right)u_2^2(x)\,d\mu(x)-\int_{\Om\times\Om} k(x,y)u_2(y)d\mu(y) \,u_2(x)d\mu(x)\over \|u_2\|_{L_\mu^2(\Om)}^2}$$
	$$={\dis\into \left(a-\frac{F(x,u_2)}{u_2}\right)u_2^2(x)\,d\mu(x)-\int_{\Om\times\Om} k(x,y)u_2(y)d\mu(y) \,u_2(x)d\mu(x)\over \|u_2\|_{L_\mu^2(\Om)}^2}$$
	and hence
	$$\into \left(-\frac{F(x,u_1)}{u_1}+\frac{F(x,u_2)}{u_2}\right)u_2^2(x)\,d\mu(x)=0.$$
	Since $u_2>0$, we have
	$$\frac{F(x,u_1)}{u_1}=\frac{F(x,u_2)}{u_2}$$
	and therefore $u_1=u_2$. This concludes the proof
\end{proof}
\begin{remark}\label{multunic}
	We emphasize that the uniqueness result obtained in the previous theorem holds only when $\Omega=\Omega_\ell$. 
	If several sets $\Omega_\ell$ are present, the result applies separately to each of them. 
	Consequently, \uline{multiple nonnegative and nontrivial solutions of \eqref{geneq2} may exist}.
	
	More precisely, for each $\Omega_\ell$ satisfying the assumptions of Theorem~6.1, one may define $u$ either as the strictly positive solution given by the theorem or as the null function on $\Omega_\ell$. 
	Hence, different solutions can be constructed by choosing, independently on each $\Omega_\ell$, either the positive solution or zero. 
	Uniqueness therefore holds only within each class $\Omega_\ell$.
	
	Similarly, for every point $x\in \cal N$, one may take $u$ such that 
	\[
	au=F(x,u),
	\]
	which typically holds at least when $u=0$, since most biological models assume $F(x,0)=0$.
	
	%In order to obtain a nontrivial solution, $u$ must be nonzero on at least one set $\Omega_\ell$ or at some point of $\cal N$.
	
	%If one chooses $u$ to be positive whenever possible, this yields the existence of a unique strictly positive solution.	
\end{remark}

We also show the following nonexistence result:
\begin{theorem}\label{thnoexistence}
	Assume that \eqref{hipk} and \eqref{hipcrF} hold. Suppose that there exists 
	$\alpha \in L^\infty_\mu(\Omega)$ such that
	\begin{equation}\label{Falfanexist}
		F(x,u) \le \alpha(x)u, \quad \text{ a.e. } x \in \Omega .
	\end{equation}
	Then, the following holds:
	\begin{enumerate}[labelsep=0.5em,label=(\roman*)]
		\item If $\lambda_p(L-\alpha I)>0$, the unique solution of \eqref{geneq2} is $u\equiv 0$ a.e. $\Om$.
		\item If $\lambda_p(L-\alpha I)=0$ and
		\beq\label{medigcero}
		\mu\!\left(\left\{(x,u(x))\in\Omega\times L_\mu^2(\Omega):
		F(x,u(x))=\alpha(x)u(x)\right\}\right)=0,
		\eeq
		the unique solution of \eqref{geneq2} is $u\equiv 0$ a.e. $\Om$.
	\end{enumerate}
\end{theorem}
\begin{proof}
	We consider $u\in L^2_\mu(\Om)$ solution of \eqref{geneq2}. Then, we have
	$$Lu=F(x,u)\Leftrightarrow (L-\alpha)u=F(x,u)-\alpha u.$$
	Thus, by multiplying for u and integrating in $\Om$, we obtain
	$$\into (L-\alpha)u^2\,d\mu(x)=\into\big(F(x,u)-\alpha u\big)u\,d\mu(x).$$
	Moreover, thanks to the definition of first spectral value given in \eqref{defautovppal}, and by applying hipothesis \eqref{Falfanexist}, it is satisfied
	$$\la_p(L-\alpha I)\into u^2\, d\mu(x)\leq \into (L-\alpha)u^2\,d\mu(x)=\into\big(F(x,u)-\alpha u\big)u\,d\mu(x)\leq0.$$
	We distinguish the following cases:
	\begin{itemize}
		\item If $\lambda_p(L-\alpha I)>0$, then
		$$0\leq\into u^2\, d\mu(x)\leq0.$$
		Thus, $u\equiv 0$ a.e. $\Om$.
		\item If $\lambda_p(L-\alpha I)=0$, then
		$$0\leq\into\big(F(x,u)-\alpha u\big)u\,d\mu(x)\leq0.$$
		Hence, if \eqref{medigcero} holds, it implies $u\equiv 0$ a.e. $\Om$.
	\end{itemize}
\end{proof}
Finally, we prove a regularity result which allows us to ensure that the solution $u$ is continuous when $k$ and $F$ are also continuous:
\begin{theorem}\label{thregularity}
	Let $\Om$ be a locally compact Hausdorff space, $\mu$ a nonnegative Radon measure in  $\Om$, $a\in C^0(\Om)$, $k\in C^0(\Om\times\Om)\cap L^\infty_{\mu\otimes\mu}(\Om\times\Om)$ symmetric, nonnegative and $F\in C^0(\Om\times [0,\infty))$. Assume that 
	\beq\label{hipco1}{\rm card}\big\{s\in [0,\infty): \ a(x)s-F(x,s)=t\big\}\leq 1,\quad \forall\, (x,t)\in \Om\times (0,\infty),\eeq
	and that for every compact set $K\subset\Om$ it is satisfied
	\beq\label{hipco2} \lim_{s\to\infty} \big(a(x)s-F(x,s)\big)=\infty\ \hbox{ uniformly  in }K.\eeq
	Then, if there exists $u\in L^1(\Om)$ nonnegative solution of
	\beq\label{ecsol}Lu=a(x)u-\into k(x,y)u(y)\,d\mu(y)=F(x,u)\ \hbox{ a.e. }\Om,\eeq
	such that
	\beq\label{intopos}\into k(x,y)u(y)d\mu(y)>0,\quad\forall x\in\Om,\eeq
	$u$ is continuous in $\Om$.
\end{theorem}
\begin{proof}
	We consider the map $G=(G_1,G_2):\Omega\times [0,\infty)\to \Omega\times\mathbb{R}$
	defined by
	\[
	G(x,s)=(x,a(x)s-F(x,s)), \qquad \forall\, (x,s)\in \Omega\times [0,\infty).
	\]
	
	Thanks to \eqref{hipco1}, by defining
	\[
	A=\big\{(x,s)\in \Omega\times [0,\infty):\ G_2(x,s)> 0\big\},
	\qquad 
	B=G(A),
	\]
	we deduce that $G$ is bijective from $A$ onto $B$, and consequently
	$G^{-1}:B\to A$ is well defined.
	
	We next prove that $G^{-1}$ is continuous on $B$. Let $(x_0,t_0)\in B$ and
	choose $s_0\geq 0$ such that
	\[
	a(x_0)s_0-F(x_0,s_0)=t_0,
	\]
	so that $(x_0,s_0)=G^{-1}(x_0,t_0)$. Let $U$ be a compact neighbourhood of $x_0$
	and let $\varepsilon>0$. To prove continuity, it suffices to show that there
	exists $\delta>0$ such that
	\begin{equation}\label{cont}
		G^{-1}\Big(\big(U\times(t_0-\delta,t_0+\delta)\big)\cap B\Big)
		\subset U\times(s_0-\varepsilon,s_0+\varepsilon).
	\end{equation}
	
	Taking
	$$\delta:=\inf\big\{|a(x)(s-s_0)-F(x,s)-F(x,s_0)|:\ x\in  U,\ |s-s_0|\geq \ep\big\}>0,$$
	let us first prove that this infimum is in fact a minimum. For this purpose, we use that by \eqref{hipco2}, there exists $M\geq0$ such that
	$$a(x)s-F(x,s)\geq 1+\delta,\quad\forall (x,s)\in U\times [0,\infty),\ \ s\geq  M.$$
	Hence,
	$$\delta=\inf\big\{|a(x)(s-s_0)-F(x,s)-F(x,s_0)|:\ x\in U,\ |s-s_0|\geq \ep,\ s\leq M\big\},$$
	and then the infimum is attained, thanks to the continuity of $a$ and $F$, and the compactness of $\{(x,s)\in U\times [0,M]:\ |s-s_0|\geq \ep\}.$

	Let $x\in U$, $t\in (t_0-\delta,t_0+\delta)$ such that $(x,t)\in B$,
	and define $s\geq 0$ by
	\[
	a(x)s-F(x,s)=t,
	\]
	so that $(x,s)=G^{-1}(x,t)$. If $|s-s_0|\geq \varepsilon$, then by the
	definition of $\delta$ we would have $|t-t_0|\geq \delta$, which contradicts
	the choice of $t$. Hence $|s-s_0|<\varepsilon$, and \eqref{cont} follows.
	
	Now, we observe that if
	$u$ is a solution of \eqref{ecsol}, then, thanks to \eqref{intopos}, we have
	\[
	(x,u(x))\in A, \quad
	(x,u(x)) = G^{-1}\Big(x,\int_\Omega k(x,y)u(y)\,d\mu(y)\Big),
	\quad \forall\,x\in \Omega.
	\]
	Thanks to the continuity of the integral term with respect to $x$, and the continuity of $G^{-1}$  we then have that $u$ is a continuous function
	on $\Omega$.
\end{proof}
Once we have proved these results, we show two biological models as examples where the results proved in this section can be applied. In both cases, we assume that $\Om$ is such that $\mu(\Om)<+\infty$ and $k\in L^\infty_{\mu\otimes\mu}(\Om\times\Om)$ satisfies \eqref{hipkeqgen}.

\subsection{Nonlocal logistic equation}
Let $\Om$ be such that $\mu(\Om)<+\infty$ and $k\in L^\infty_{\mu\otimes\mu}(\Om\times\Om)$ satisfying \eqref{hipkeqgen}. We consider the following problem of the form \eqref{geneq2}:
\beq\label{eqlog}Lu=a(x)u(x)-\into k(x,y)u(y)\,d\mu(y)=\lambda u-\nu u^p,\quad u\geq0,\ \hbox{ a.e. in }\Om,\eeq
where $\la\in L^\infty_\mu(\Om)$, $\nu\in L^\infty_\mu(\Om)$, $\operatorname*{ess\,inf}\, \nu>0$ and $p>1$.
In these conditions, a characterization of existence depending on $\la_p(L-\la I)$ holds in each $\Om_l\subset\Om$ given by Definition \ref{omlN}.
\begin{theorem}\label{TeoElo}
	There exists a unique strictly positive solution of \eqref{eqlog} in each $\Om_l\subset\Om$, $l\geq1$, if and only if $\la_p(L-\la I)<0$.
\end{theorem}
\begin{proof}
	Clearly, $\la u-\nu u^p$ satisfies (\ref{Fnoneg0}),\eqref{hipsupsol} and \eqref{hipunicidad}. Moreover, by taking $\alpha=\lambda$ we also have \eqref{hipsubsol} and \eqref{Falfanexist}. Therefore, the result is a direct consequence of Theorems \ref{thexistence} and \ref{thnoexistence}.
\end{proof}
Furthermore, in $\cal N$ \eqref{eqlog} reduces to solving the equation
$$au(x)=\lambda(x) u(x)-\nu(x) u^p (x),\quad\hbox{a.e. }x\in\cal N.$$
In this case, the solution is 
$$u(x)=\left\{\ba{ll}\dis \left({\lambda(x)-a(x)\over \nu(x)}\right)^{1/(p-1)} & \hbox{ if }\lambda(x)-a(x)>0\\ \ecart\dis 0 &\hbox{ otherwise.}\ea\right.$$

We can apply Theorem \ref{thregularity} to obtain the following regularity result:
\begin{theorem}\label{threglog}
	We assume the hypothesis imposed at the begining of the section together with $\Om$  Hausdorff, locally compact, $\la\in C^0(\Om)$, $\nu\in C^0(\Om)$, $k\in C^0(\Om\times\Om)$ and $\mu$ a nonnegative Radon measure. We consider $\Om_l$, $l\geq 1$, the equivalence classes associated to the relation given by Definition \ref{defCEd1}. Then, in every $\Om_l\subset\Om$ such that there exists a positive solution $u$ in $\Om_l$, such solution satisfies $u\in C^0(\Om)$.
\end{theorem}
\begin{proof}
	In order to prove that (\ref{hipco1}) is satisfied, we take $G_2:\Om\times [0,\infty)\to\RR$ as the second component of $G(x,s)=\Big(x,\big(a(x)-\la(x)\big)s+\nu(x) s^p\Big)$,
	and we observe that
	$$\partial_s G_2(x,s)=a(x)-\la(x)+p\nu(x)s^{p-1}.$$
	Therefore, for $x$ fixed we have:
	
	If $a(x)-\la(x)\geq0$, then $G_2(x,.)$ is strictly increasing. 
	
	Otherwise, $G_2(x,.)$ satisfies
	$$\ba{c} G_2(x,0)=0,\quad G_2(x,.)\hbox{ strictly decreasing in  }(0,s_x],\\ \ecart\dis G_2(x,.)\hbox{ strictly increasing in  }[s_x,\infty),\quad\lim_{s\to\infty}G(x,s)=\infty,\ea$$
	with
	$$s_x:=\left(\frac{\la(x)-a(x)}{p\nu}\right)^{\frac{1}{p-1}}.$$
	Therefore there exists $r_x>s_x$ such that $G_2(x,.)<0$ in $(0,r_x),$ $G_2(x,.)>0$ and strictly increasing in $(r_x,\infty)$. 
	
	Thus in both cases, for every $t>0$ and every $x\in \Om$, there exists a unique $s>0$ such that $G_2(x,s)=t$.
	
	Furthermore, let $K\subset\Om$ compact and $x\in K$. We have
	$$\big(a(x)-\la(x)\big)s+\nu(x) s^p\geq \Big(\min_{ K}\big(a-\la\big)s+\min_{K}\nu s^p\Big)\to \infty\ \hbox{ if }s\to\infty,$$
	so \eqref{hipco2} also holds.
	
	Finally, if we are in the conditions of Theorem \ref{TeoElo}, there exists a unique strictly positive solution $u\in L^2(\Om_l)$ which will satisfy 
	$$\int_{\Om_l}k(x,y)u(y)d\mu(y)>0.$$
	Then, Theorem \ref{thregularity} proves  $u\in C^0(\Om_l)$.
\end{proof}
We present some numerical simulations with this model (as it is one of the most classical models in the literature) in order to illustrate the large variety of problems covered by our framework. The simplest case to consider corresponds to the finite-dimensional setting. In particular, we have 
\begin{example}[Discrete case]
	Assume that $\Omega=\{1,2,\dots,n\}$ is a finite set. 
	We consider the discrete Radon measure given by $\mu(\{i\})=1$ for every $i=1,\dots,n$. 
	In this setting, any function $u:\Omega\to\mathbb{R}$ can be identified with a vector in $\mathbb{R}^n$, and the kernel 
	$k$ can be represented by a matrix $k=(k_{ij})\in\mathcal{M}^{n\times n}$.
	
	If the matrix $k$ is stochastic and symmetric, namely $k_{ij}=k_{ji}\ge0$ and 
	$\sum_{j=1}^n k_{ij}=1$, then $k_{ij}$ represents the probability of moving 
	from $i$ to $j$ per unit time, and the associated dispersal process is a stationary Markov chain with transition matrix $k$. 
	
	For each node $i$ (which may represent, for instance, a city or a spatial patch), 
	the equation becomes
	\[
	a(i)u(i)-\sum_{j=1}^n k(i,j)u(j)
	=
	\lambda(i)u(i)-\nu(i)u^2(i),
	\qquad i=1,\dots,n.
	\]
	
	Denoting
	\[
	a_i:=a(i), \quad 
	u_i:=u(i), \quad 
	\lambda_i:=\lambda(i), \quad 
	\nu_i:=\nu(i),
	\]
	the problem reduces to solving the nonlinear system
	\[
	a_i u_i-\sum_{j=1}^n k_{ij}u_j
	=
	\lambda_i u_i-\nu_i u_i^2,
	\qquad i=1,\dots,n.
	\]
	
	In matrix form, this can be written as
	$$\left(\ba{c}
	(a_1-\la_1)u_1\\\vdots\\(a_n-\la_n)u_n
	\ea\right)-
	\left(\ba{ccc}
	k_{11}&\hdots&k_{1n}\\
	\vdots&\ddots&\vdots\\
	k_{n1}&\hdots&k_{nn}
	\ea\right)
	\left(\ba{c}
	u_1\\\vdots\\u_n
	\ea\right)=-\left(\ba{c}
	\nu_1u^2_1\\\vdots\\\nu_nu^2_n
	\ea\right)$$
	
	Therefore, in the discrete framework the nonlocal problem reduces to a 
	finite-dimensional nonlinear system, which can be interpreted as a logistic-type 
	equation on a weighted graph. 
	
	For instance, to illustrate this, we consider a particular case where $n=8$ and the transition matrix $k$ is of the form
	$$k=\left(
	\ba{cccccccc}
	k_{11} & k_{12} & 0 & 0 & 0 & 0 & 0 & 0\\
	k_{12} & k_{22} & k_{23} & 0 & 0 & k_{26} & 0 & 0\\
	0 & k_{23} & k_{33} & 0 & 0 & 0 & k_{37} & 0\\
	0 & 0 & 0 & k_{44} & 0 & 0 & 0 & k_{48}\\
	0 & 0 & 0 & 0 & 0 & 0 & 0 & 0\\
	0 & k_{26} & 0 & 0 & 0 & k_{66} & 0 & 0\\
	0 & 0 & k_{37} & 0 & 0 & 0 & k_{77} & 0\\
	0 & 0 & 0 & k_{48} & 0 & 0 & 0 & k_{88}
	\ea
	\right).$$
	The system can be visualized through the associated 
	state graph:
	
	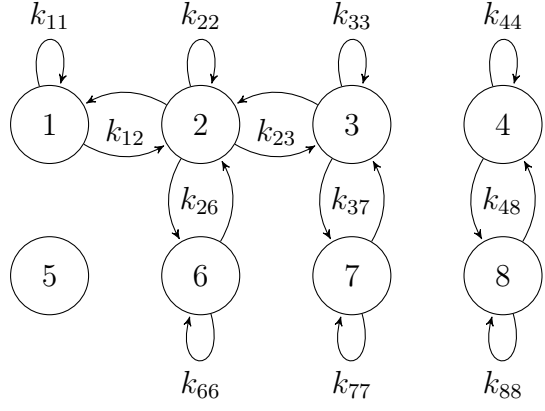
\begin{figure}[H] 
		\centering 
		\begin{tikzpicture}[->, %Líneas direccionadas
			>=stealth', %Puntas de flecha rellenas
			shorten >=1pt,
			auto,
			node distance=2cm,] %distancia mínima entre nodos
			
			\node[state] (1) {1};
			\node[state,right of = 1] (2) {2};
			\node[state,right of = 2] (3) {3};
			\node[state,right of = 3] (4) {4};
			\node[state,below of = 1] (5) {5};
			\node[state,right of = 5] (6) {6};
			\node[state,right of = 6] (7) {7};
			\node[state,right of = 7] (8) {8};
			
			\draw (1) edge[loop above] node{$k_{11}$} (1);
			\draw (2) edge[loop above] node{$k_{22}$} (2);
			\draw (3) edge[loop above] node{$k_{33}$} (3);
			\draw (4) edge[loop above] node{$k_{44}$} (4);
			
			%\draw (5) edge[loop below] node{$k_{55}$} (5);
			\draw (6) edge[loop below] node{$k_{66}$} (6);
			\draw (7) edge[loop below] node{$k_{77}$} (7);
			\draw (8) edge[loop below] node{$k_{88}$} (8);
			
			\draw (1) edge[bend right,above] node{$k_{12}$} (2);
			\draw (2) edge[bend right] (1);
			\draw (2) edge[bend right,above] node{$k_{23}$} (3);
			\draw (3) edge[bend right] (2);
			
			\draw (2) edge[bend right,right] node{$k_{26}$} (6);
			\draw (6) edge[bend right] (2);
			\draw (3) edge[bend right,right] node{$k_{37}$} (7);
			\draw (7) edge[bend right] (3);
			\draw (4) edge[bend right,right] node{$k_{48}$} (8);
			\draw (8) edge[bend right] (4);
		\end{tikzpicture}
		\caption{States diagram}
	\end{figure}
	In this case, 
	$$\Om=\mathcal{N}\cup\Om_1\cup\Om_2,$$
	where $\mathcal{N}=\{5\}$, $\Om_1=\{1,2,3,6,7\}$ and $\Om_2=\{4,8\}$. Thanks to Theorem \ref{Thscdes}, the problem can be solved independently in $\mathcal{N}$, $\Om_1$ and $\Om_2$. Therefore, the problem is decomposed as follows:
	$$L_{\mathcal{N}}-\la_\mathcal{N}I:=(a_5-\la_5)u_5=-\nu_5 u_5^2,$$
	$$L_{\Om_1}-\la_{\Om_1}I:=\left(\ba{l}
	(a_1-\la_1)u_1\\
	(a_2-\la_2)u_2\\
	(a_3-\la_3)u_3\\
	(a_6-\la_6)u_6\\
	(a_7-\la_7)u_7
	\ea\right)
	-
	\left(\ba{ccccc}
	k_{11}&k_{12}&0&0&0\\
	k_{12}&k_{22}&k_{23}&k_{26}&0\\
	0&k_{23}&k_{33}&0&k_{37}\\
	0&k_{26}&0&k_{66}&0\\
	0&0&k_{37}&0&k_{77}
	\ea\right)
	\left(\ba{l}
	u_1\\
	u_2\\
	u_3\\
	u_6\\
	u_7
	\ea\right)=-\left(\ba{l}\dis
	\nu_1u^2_1\\\dis
	\nu_2u^2_2\\\dis
	\nu_3u^2_3\\\dis
	\nu_6u^2_6\\\dis
	\nu_7u^2_7\\
	\ea\right),$$
	$$L_{\Om_2}-\la_{\Om_2}I:=\left(\ba{l}
	(a_4-\la_4)u_4\\
	(a_8-\la_8)u_8
	\ea\right)
	-
	\left(\ba{cc}
	k_{44}&k_{48}\\
	k_{48}&k_{88}
	\ea\right)
	\left(\ba{l}
	u_4\\
	u_8
	\ea\right)=-\left(\ba{l}\dis
	\nu_4u^2_4\\\dis
	\nu_8u^2_8
	\ea\right).$$
	
	\noindent\textbf{Case A}
	
	By assuming, for instance $a=(1,1.2,0.8,0.6,1,0.9,0.7,0.3)$, $\nu=(1,1,1,1,1,1,1,1),$ $\la=(2,-0.5,1.5,2,2,1.2,-0.3,-0.1)$ and
	$$k=\left(
	\ba{cccccccc}
	0.7 & 0.3 & 0 & 0 & 0 & 0 & 0 & 0\\
	0.3 & 0.4 & 0.2 & 0 & 0 & 0.1 & 0 & 0\\
	0 & 0.2 & 0.55 & 0 & 0 & 0 & 0.25 & 0\\
	0 & 0 & 0 & 0.65 & 0 & 0 & 0 & 0.35\\
	0 & 0 & 0 & 0 & 0 & 0 & 0 & 0\\
	0 & 0.1 & 0 & 0 & 0 & 0.9 & 0 & 0\\
	0 & 0 & 0.25 & 0 & 0 & 0 & 0.75 & 0\\
	0 & 0 & 0 & 0.35 & 0 & 0 & 0 & 0.65
	\ea
	\right),$$
	we can solve the equation for $u_5$ to obtain the nonnegative solutions $u_5=0$ and $u_5=1$. Furthermore, computing the principal eigenvalues for $\Omega_1$ and $\Omega_2$ yields
	\begin{itemize}
		\item $\la_p(L_{\Om_1}-\la_{\Om_1}I)\simeq-1.7308<0,$
		\item $\la_p(L_{\Om_2}-\la_{\Om_2}I)\simeq-2.1157<0.$
	\end{itemize}
	Therefore, according to Theorem \ref{TeoElo}, the problem admits unique strictly positive solutions on $\Om_1$ and $\Om_2$. 
	
	To verify this, we compute the sub and supersolutions method by solving a point-fixed problem in each step. Taking a constant supersolution as initialization, we obtain
	\begin{itemize}
		\item In $\Om_1:$ $\om_1:=(u_1,u_2,u_3,u_6,u_7)\simeq(1.787, 0.5181, 1.409, 1.2417, 0.4815)>0.$
		\item In $\Om_2:$ $\om_2:=(u_4,u_8)\simeq(2.2105,1.0134)>0$.
	\end{itemize}
	We represent the unique strictly positive solution:
	\begin{figure}[H]
		\centering
		\includegraphics[width=0.6\textwidth]{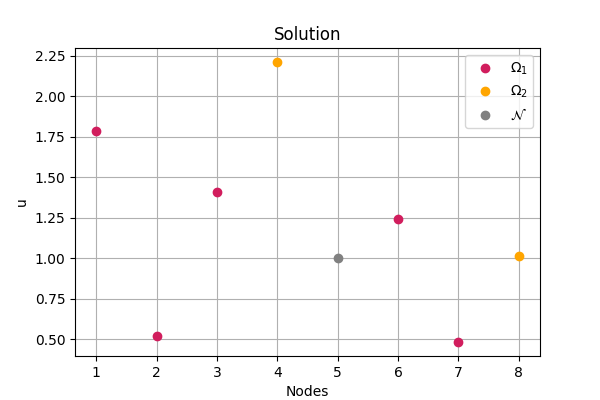}
		\caption{Strictly positive solution of the discrete logistic equation with a stationary Markov transition kernel. The domain is decomposed into equivalence classes with negative principal eigenvalue.}
	\end{figure}
	We observe that as previously mentioned, multiple nonzero solutions can be obtained by choosing $u_5$, $\om_1$ and $\om_2$ either as zero or as the strictly positive solution, with at least one of them being strictly positive.
	
	\noindent\textbf{Case B}
	
	To conclude this example, we now assume that $\la_4=\la_8=-1$, while $a$. $\nu$, $k$ and the remaining components of $\la$ are defined as in Case A. In this case, we obtain
	$$\la_p(L_{\Om_2}-\la_{\Om_2}I)\simeq0.4192>0.$$
	Solving the problem in $\Om_2$, we find
	$$\om_2=(u_4,u_8)\simeq(1.5656\times 10^{-10}, 2.3743\times 10^{-10})\simeq(0,0).$$
	Hence, when $\la_p(L_{\Om_2}-\la_{\Om_2})<0$, we have been able to obtain a strictly positive solution on $\Om_2$. In contrast, when $\la_p(L_{\Om_2}-\la_{\Om_2})>0$, the unique nonnegative solution is the null function. This behavior is fully consistent with Theorem \ref{TeoElo}.
	
	For the sake of clarity, Table \ref{tab:summary_results} collects the results for each of the regions discussed above across the analyzed cases.
	\begin{table}[h]
		\centering
		\setlength{\tabcolsep}{3pt}
		\renewcommand{\arraystretch}{1.1}
		\begin{tabular}{c|c|c|c}
			\textbf{Region} & \textbf{Case} & $\boldsymbol{\lambda_p}$ & \textbf{Choices} \\ \hline
			$\Omega_1$ & - & -1.7308 & $\big\{0,(1.787, 0.5181, 1.409, 1.2417, 0.4815)\big\}$ \\ 
			$\Omega_2$ & A & -2.1157 & $\big\{0,(2.2105,1.0134)\big\}$  \\ 
			$\Omega_2$ & B & +0.4192 & $\big\{0\big\}$  \\ 
			$\mathcal{N}$ & - & -1 & $\big\{0,1\big\}$  \\ 
		\end{tabular}
		\caption{Summary of the results in the different regions, depending on the sign of the principal eigenvalue in $\Omega_2$. We also report the corresponding sets of admissible choices leading to multiple positive and nontrivial solutions.}
		\label{tab:summary_results} % <--- Movido abajo del caption para que no falle el hipervínculo
	\end{table}
\end{example}
Another interesting case is the most commonly found in literature, which corresponds to the Lebesgue's measure and $\Om\subset\RR^N$. We study the following example:
\begin{example}[Continuous case]
	Let $\Omega=\mathbb{R}$. An interesting situation arises when 
	$k \in C^0(\overline{\Omega}\times\overline{\Omega})$ 
	is a symmetric convolution kernel supported in $(-c,c)$, with $c>0$. 
	For $0<s<c$, we define
	\beq\label{convker}
	k(x,y)=J(x-y):=\frac{3}{4}\frac{\big(s^2-(x-y)^2\big)^+}{s^3},
	\eeq
	which has unit integral and decreases with the distance of $x$ and $y$.
	
	Moreover, we assume $\gamma>0$, $\lambda \in C^0(\overline{\Omega})$, 
	and $a \in C^0(\overline{\Omega})$ given by
	\beq\label{contlamu}
	\lambda(x)=\frac{\gamma}{1+\gamma x^2}, 
	\qquad 
	a(x)=\hat{k}(x)=\int_{\Omega} k(x,y)\,dy=1,
	\qquad 
	\nu(x)=1 
	\quad \text{for all } x\in\overline{\Omega}.
	\eeq
	This $\la$ concentrates the population's growth rate around $x=0$.
	Within this framework, we consider the problem
	\beq\label{contlogeq}
	Lu:=\big(1-\lambda(x)\big)u(x)
	-\int_{\Omega} J(x-y)u(y)\,dy
	=-u(x)^2.
	\eeq
	
	Since every point of the domain is connected through the support of the kernel, 
	the maximum principle holds. In the particular case where $c=1$, $s=0.01$ and $\ga=20$, computing the principal eigenvalue  through a power iteration scheme, we obtain
	\[
	\lambda_p(L)\simeq -19.938 <0.
	\]
	
	Applying the sub-- and supersolution method, initialized with a constant 
	supersolution, we obtain the iterations and the corresponding solution 
	depicted in Figure~\ref{fig:continuous_example}.
	
	\begin{figure}[H]
		\centering
		\includegraphics[width=0.49\textwidth]{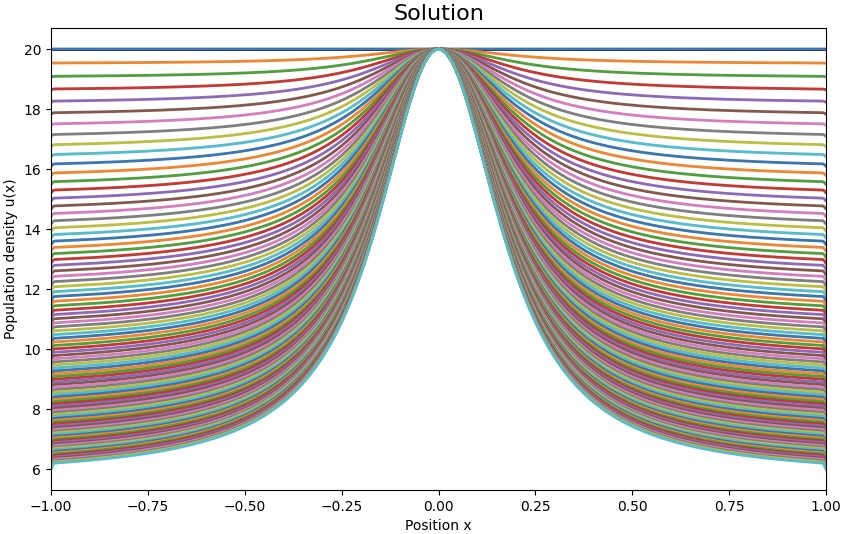}
		\hfill
		\includegraphics[width=0.49\textwidth]{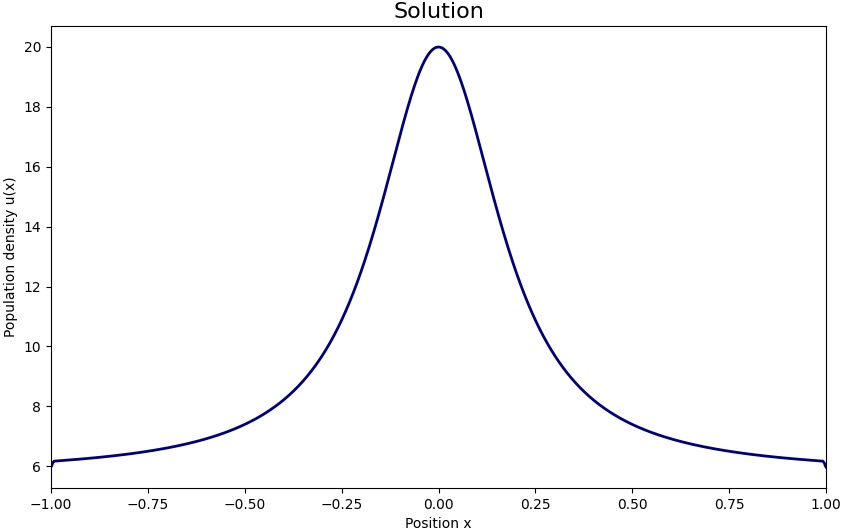}
		\caption{Sub- and supersolution iterations starting from a constant supersolution of the problem defined in \eqref{convker}--\eqref{contlogeq}, together with the resulting strictly positive solution.}
		\label{fig:continuous_example}
	\end{figure}
	
	The method converges to a strictly positive solution, in agreement with 
	the sign of the principal eigenvalue. Moreover, as it was hoped, the population's density is concentrated where $\la$ takes larger values.
	
	We repeat the test with the same linear operator $L$ and the same values for $c$, $s$ and $\ga$ to illustrate, for instance, the influence of the parameter $p$ for the logistic equation
	$$Lu=-u^p$$
	We have already studied the case where $p=2$. We compute the method of sub and supersolutions for larger values of $p$ and we compare the results obtained by representing them in the following graphs:
	
	\begin{figure}[H]
		\centering
		\includegraphics[width=0.49\textwidth,height=4.75cm]{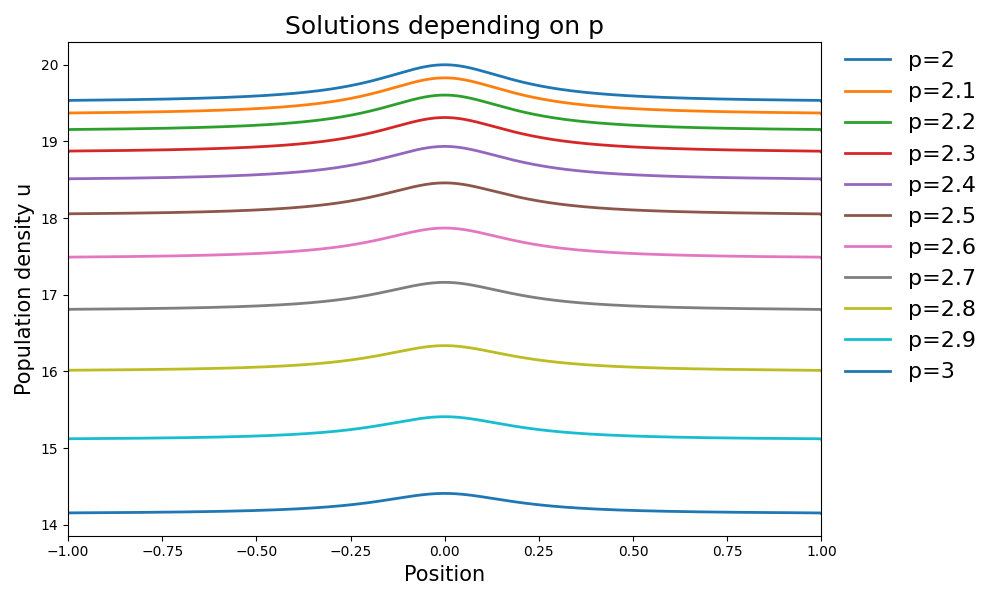}
		\hfill
		\includegraphics[width=0.49\textwidth,height=4.755cm]{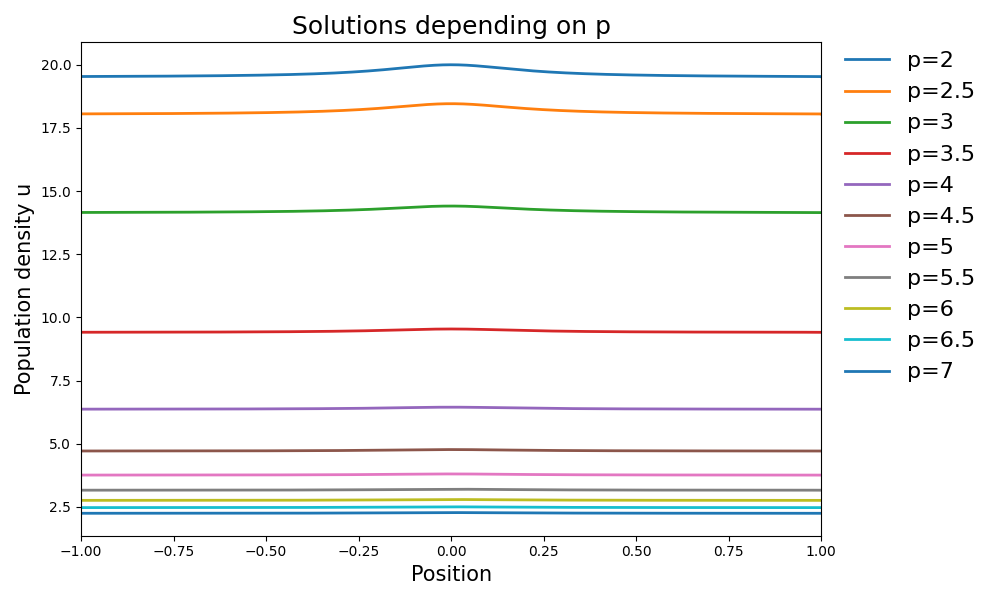}
		\caption{Comparison of solutions of the logistic equation for varying values of the parameter $p$. The left figure corresponds to $p=2,2.1,2.2,\ldots,3$, while the figure panel corresponds to $p=2,2.5,3,\ldots,7$.}
		\label{fig:continuous_example2}
	\end{figure}
	We observe the solutions are smaller the larger the value of $p$ chosen.
\end{example}
Another interesting situation in our framework is the interaction between species living in domains with different measures. We highlight two illustrative cases:
\begin{example}[A natural reserve divided by a river]
	We consider $\Om=(-1,1)\times(0,1)$ as an ecosystem (such as a natural reserve) divided by a resource supplier, a river for example, defined by $\Ga=\{(x,y)\in\Om:x=0\}$. Let $\mu$ be the Radon measure given by
	$$\mu=\mathcal{H}_{1|\Ga}+\mathcal{H}_2,$$
	where $\mathcal{H}_{1|\Ga}$ is the one-dimensional Hausdorff measure restricted to $\Ga$ and $\mathcal{H}_2$ is the two-dimensional Hausdorff measure.
	
	Moreover, we assume $k\in L_{\mu\otimes\mu}^\infty(\Om\times\Om)$ defined as
	\beq\label{kernel_rio}
	k\big((x_1,x_2),(y_1,y_2)\big)=\left\{
	\ba{l}
	k_1\big((x_1,x_2),(y_1,y_2)\big),\quad\hbox{if }x_1\neq0,y_1\neq0,\\\ecart\dis
	k_2\big((x_1,x_2),(0,y_2)\big),\quad\hbox{if }x_1\neq0,y_1=0,\\\ecart\dis
	k_3\big((0,x_2),(0,y_2)\big),\quad\hbox{if }x_1=0,y_1=0.
	\ea
	\right.
	\eeq
	
	where $k_1\in  L_{\mathcal{H}_2\otimes \mathcal{H}_2}^\infty(\Om\times\Om)$, $k_2\in L_{\mathcal{H}_2\otimes \mathcal{H}_{1|\Ga}}^\infty(\Om\times\Gamma)$, $k_3\in L_{\mathcal{H}_{1|\Ga}\otimes \mathcal{H}_{1|\Ga}}^\infty(\Ga\times\Ga)$.
	
	Let $\la\in L^\infty_\mu(\Om)$, $\nu\in L^\infty_\mu(\Om)$ with $\mu\big(\{\nu=0\}\big)=0$, we look for $u\in L^2_\mu(\Om)$ solution of
	\beq\label{rio_sist1}
	\left\{
	\ba{c}
	\dis\int_{\RR^2} k(x,y)(u(x)-u(y))d\mu(y)=\la u-\nu u^2,\quad x\in\Om,\\\ecart\dis
	u = 0 \ \hbox{ in }\ \RR^2\setminus\Om.
	\ea
	\right.
	\eeq 
	We can write $\la,\nu$ as
	$$\la=\chi_{\Ga}\la_\Ga+\chi_{\Om\setminus\Ga}\la_\Om,\quad\nu=\chi_{\Ga}\nu_\Ga+\chi_{\Om\setminus\Ga}\, \nu_\Om.$$
	Problem \eqref{rio_sist1} is equivalent to find $u\in L_{\mathcal{H}_2}^2(\Om)$, $\tilde u\in L_{\mathcal{H}_{1|\Ga}}^2(\Ga)$ such that
	\beq\label{rio_sist2}
	\left\{
	\ba{l}
	\dis\int_{\RR^2} k_1(x,y)(u(x)-u(y))d\mathcal{H}_2(y)\\\ecart\dis
	\hskip2cm+\int_{\RR}k_2(x,y)(u(x)-\tilde u(y))d\mathcal{H}_{1|\Ga}(y)=\la_\Om u-\nu_\Om u^2,\quad x\in\Om,\\\ecart\dis
	\dis\int_{\RR^2} k_2(y,x)(\tilde u(x)-u(y))d\mathcal{H}_2(y)\\\ecart\dis
	\hskip2cm+\int_{\RR}k_3(x,y)(\tilde u(x)-\tilde u(y))d\mathcal{H}_{1|\Ga}(y)=\la_{\Ga}u-\nu_\Ga u^2,\quad x\in\Ga,\\\ecart\dis
	u = 0 \ \hbox{ in }\ \RR^2\setminus\Om,\quad \tilde u = 0 \ \hbox{ in }\ \RR\setminus\Ga.
	\ea
	\right.
	\eeq 	
	For the numerical tests, we consider positive parameters 
	$\alpha$, $\beta$, $\gamma$, $a$, $b$, and $c$. 
	The kernels $k_1$, $k_2$, and $k_3$ in \eqref{kernel_rio} are chosen as
\begin{equation}\label{kernel_rio2}
			\begin{aligned}
					\dis k_1\big((x_1,x_2),(y_1,y_2)\big) &= e^{\dis-\alpha\big((x_1-y_1)^2+(x_2-y_2)^2\big)},\\\ecart\dis
					k_2\big((x_1,x_2),(0,y_2)\big) &= e^{\dis-\beta\big(x_1^2+(x_2-y_2)^2\big)},\\\ecart\dis
					k_3\big((0,x_2),(0,y_2)\big) &= e^{\dis-\gamma(x_2-y_2)^2}.
				\end{aligned}
		\end{equation}
	
	Moreover, we assume $\nu_\Omega=\nu_\Gamma=\lambda_{\Gamma}=1$ and
	\[
	\lambda_\Omega(x,y)=\frac{1}{1+d x^2}, \qquad d>0.
	\]
	
	We fix the parameter values $\alpha=15$, $\beta=10$ and $\gamma=20$. We compute the solutions $u$ and $\tilde u$ of \eqref{rio_sist2} for 
	$d=1$, $d=5$, and $d=10$. The corresponding numerical results are shown in 
	Figures~\ref{fig:continuous_example3}, \ref{fig:continuous_example4}, and 
	\ref{fig:continuous_example5}, respectively.
	\begin{figure}[H]
			\centering
			\includegraphics[width=0.525\textwidth]{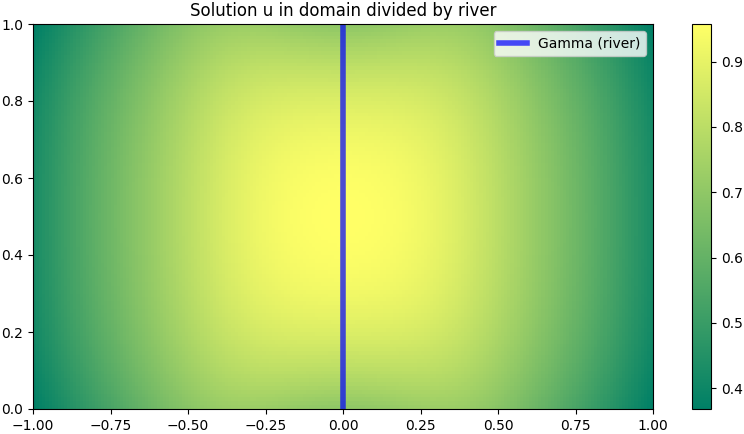}
			\hfill
			\includegraphics[width=0.465\textwidth]{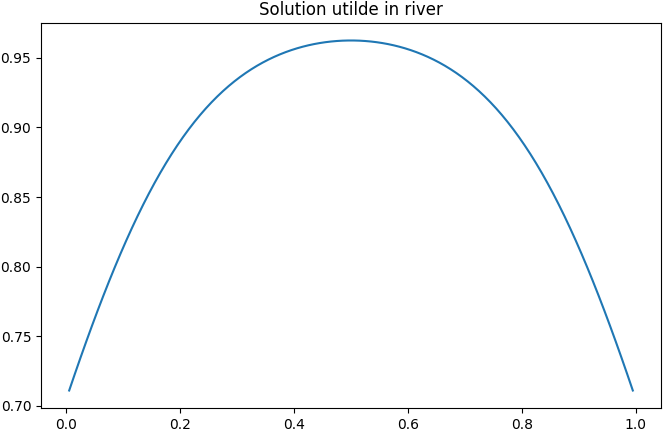}
			\caption{Solutions $u$ and $\tilde u$ of \eqref{rio_sist2} for $d=1$.}
			\label{fig:continuous_example3}
		\end{figure}
	\begin{figure}[H]
			\centering
			\includegraphics[width=0.525\textwidth]{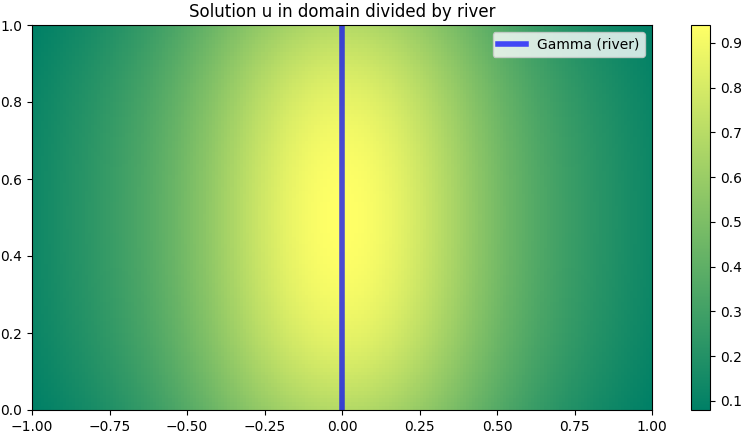}
			\hfill
			\includegraphics[width=0.465\textwidth]{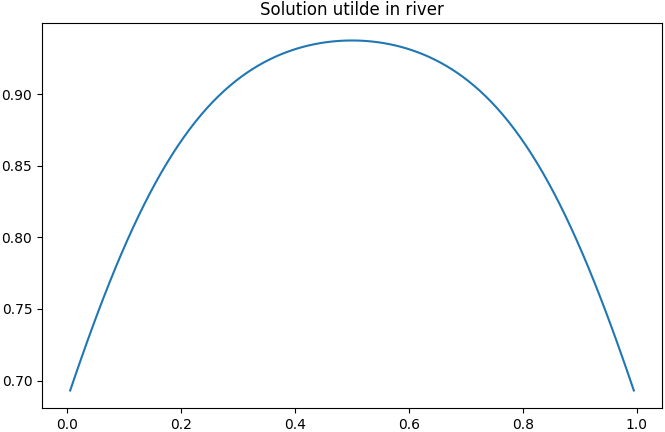}
			\caption{Solutions $u$ and $\tilde u$ of \eqref{rio_sist2} for $d=5$.}
			\label{fig:continuous_example4}
		\end{figure}
	\begin{figure}[H]
			\centering
			\includegraphics[width=0.525\textwidth]{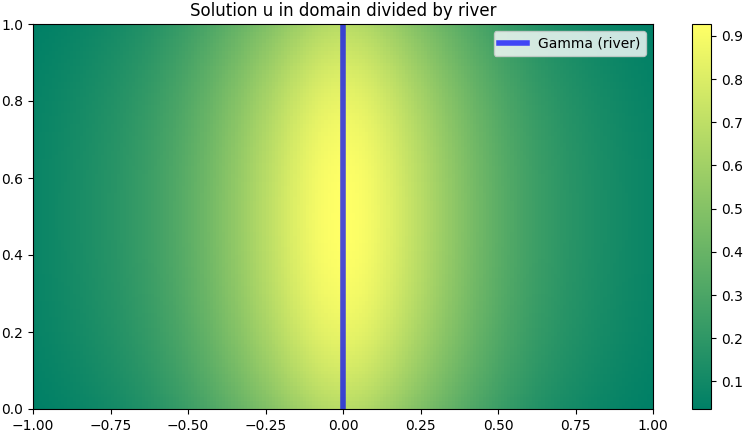}
			\hfill
			\includegraphics[width=0.465\textwidth]{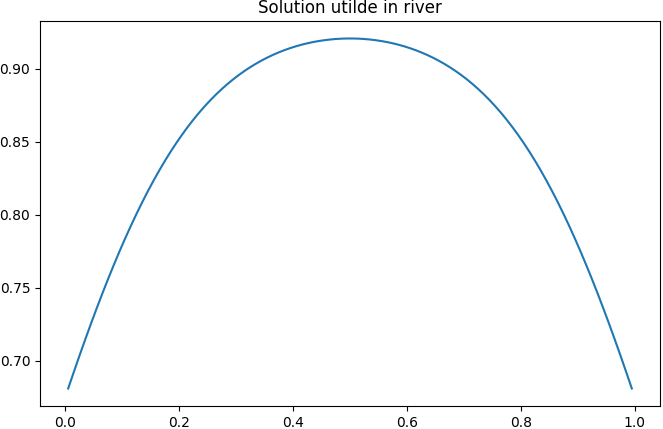}
			\caption{Solutions $u$ and $\tilde u$ of \eqref{rio_sist2} for $d=10$.}
			\label{fig:continuous_example5}
		\end{figure}
	We observe that, since the parameters have been chosen so that the principal eigenvalue is negative, the computed solutions are strictly positive.
	
	Moreover, the largest values of the solution in $\Omega$ are attained in the vicinity of the river. This is expected, as the river is taken into account in the model despite having lower dimensionality. This behaviour is consistent with ecological observations: species typically concentrate near rivers, where resources are more abundant. 
	
	Furthermore, due to the definition of $\lambda_\Omega$, increasing the parameter $d$ reduces the size of the region around the river where the species density is highest. 
	As a consequence, the overall magnitude of the solutions $u$ and $\tilde u$ decreases slightly as $d$ increases.
\end{example}

To finish we present a last example:
\begin{example}[Two cities in a rural area]
	We consider the domain $\Om=(-1,1)\times(0,1)$ which represents an ecosystem with two resource suppliers located at the points $\mathcal{C}_1=(-0.75,0.25)$ and $\mathcal{C}_2=(0.75,0.75)$. For instance, this situation may model two cities in a rural area, a desert with two oases, a natural reserve with two ponds, etc. Throughout this example, we refer to $\mathcal{C}_1$ and $\mathcal{C}_2$ as cities. 
	
	Let $\mu$ be the Radon measure given by
	$$\mu=\mu_{\mathcal{C}_1}+\mu_{\mathcal{C}_2}+\mathcal{H}_2,$$
	where $\mu_{\mathcal{C}_i}(\{\mathcal{C}_i\})=1$, $i=1,2$, and 0 elsewhere. As before, $\mathcal{H}_2$ denotes the two-dimensional Hausdorff measure.
	
	Moreover, we assume that $k\in L_{\mu\otimes\mu}^\infty(\Om\times\Om)$ is defined as
	\beq\label{kernel_ciudades}
	k(x,y)=\left\{
	\ba{l}
	k_1(x,y),\quad\hbox{if }x,y\notin\{\mathcal{C}_1,\mathcal{C}_2\},\\\ecart\dis
	k_2(x,\mathcal{C}_1),\quad\hbox{if }x\notin\{\mathcal{C}_1,\mathcal{C}_2\},\ y=\mathcal{C}_1,\\\ecart\dis
	k_3(x,\mathcal{C}_2),\quad\hbox{if }x\notin\{\mathcal{C}_1,\mathcal{C}_2\},\ y=\mathcal{C}_2,\\\ecart\dis
	k_{11},\quad\hbox{if }x=y=\mathcal{C}_1,\\\ecart\dis
	k_{22},\quad\hbox{if }x=y=\mathcal{C}_2,\\\ecart\dis
	k_{12},\quad\hbox{if }x=\mathcal{C}_1,\quad y=\mathcal{C}_2.
	\ea
	\right.
	\eeq	
	where $k_1\in L_{\mathcal{H}_2\otimes \mathcal{H}_2}^\infty(\Omega\times\Omega)$ is symmetric, 
	$k_2\in L_{\mathcal{H}_2\otimes \mu_{\mathcal{C}_1}}^\infty(\Omega\times\{\mathcal{C}_1\})$, 
	$k_3\in L_{\mathcal{H}_2\otimes \mu_{\mathcal{C}_2}}^\infty(\Omega\times\{\mathcal{C}_2\})$, 
	and $k_{11},k_{12},k_{22}>0$.
	
	As in the previous example, let 
	$\lambda\in L^\infty_\mu(\Omega)$ and 
	$\nu\in L^\infty_\mu(\Omega)$ with 
	$\mu\big(\{\nu=0\}\big)=0$. 
	We look for a function $u\in L^2_\mu(\Omega)$ satisfying
	\beq\label{ciudades_sist1}
	\left\{
	\ba{c}
	\dis\int_{\RR^2} k(x,y)(u(x)-u(y))d\mu(y)=\la u-\nu u^2,\quad x\in\Om,\\\ecart\dis
	u = 0 \ \hbox{ in }\ \RR^2\setminus\Om.
	\ea
	\right.
	\eeq 
	We can decompose $\lambda$ and $\nu$ as
	$$\la=\chi_{\mathcal{C}_1}\la_{\mathcal{C}_1}+\chi_{\mathcal{C}_2}\la_{\mathcal{C}_2}+\chi_{\Om\setminus\{\mathcal{C}_1,\mathcal{C}_2\}}\la_\Om,\quad\nu=\chi_{\mathcal{C}_1}\nu_{\mathcal{C}_1}+\chi_{\mathcal{C}_2}\nu_{\mathcal{C}_2}+\chi_{\Om\setminus\{\mathcal{C}_1,\mathcal{C}_2\}}\, \nu_\Om.$$
	Problem \eqref{ciudades_sist1} is equivalent to finding 
	$u\in L^2_{\mathcal{H}_2}(\Omega)$ together with scalar unknowns 
	$\tilde u_1,\tilde u_2\in\mathbb{R}$ corresponding to the atomic components of the measure,
	such that the following coupled system holds:
	\beq\label{ciudades_sist2}
	\left\{
	\ba{l}
	\dis\int_{\RR^2} k_1(x,y)(u(x)-u(y))d\mathcal{H}_2(y)+k_2(x,\mathcal{C}_1)(u(x)-\tilde u_1)\\\ecart\dis
	\hskip3cm+k_3(x,\mathcal{C}_2)(u(x)-\tilde u_2)=\la_\Om u-\nu_\Om u^2,\quad x\in\Om,\\\ecart\dis
	\dis\int_{\RR^2} k_2(y,\mathcal{C}_1)(\tilde u_1-u(y))d\mathcal{H}_2(y)+k_{12}(\tilde{u}_1-\tilde{u}_2)=\la_{\mathcal{C}_1}\tilde u_1-\nu_{\mathcal{C}_1}\tilde u_1^2,\\\ecart\dis
	\dis\int_{\RR^2} k_3(y,\mathcal{C}_2)(\tilde u_2-u(y))d\mathcal{H}_2(y)+k_{12}(\tilde{u}_2-\tilde{u}_1)=\la_{\mathcal{C}_2}\tilde u_2-\nu_{\mathcal{C}_2}\tilde u_2^2,\\\ecart\dis
	u = 0 \ \hbox{ in }\ \RR^2\setminus\Om.
	\ea
	\right.
	\eeq 
	For the numerical tests, we consider positive parameters 
	$\alpha$, $\beta$, $\gamma$, $a$, $b$, and $c$. 
	The kernels $k_1$, $k_2$, and $k_3$ in \eqref{kernel_ciudades} are chosen as
	\begin{equation}\label{kernel_ciudades2}
		\begin{aligned}
			k_1\big((x_1,x_2),(y_1,y_2)\big) &= \alpha\big(a^2-(x_1-y_1)^2-(x_2-y_2)^2\big),\\
			k_2\big((x_1,x_2),(-0.75,0.25)\big) &= \beta\big(b^2-(x_1+0.75)^2-(x_2-0.25)^2\big),\\
			k_3\big((x_1,x_2),(0.75,0.75)\big) &= \ga\big(c^2-(x_1-0.75)^2-(x_2-0.75)^2\big).
		\end{aligned}
	\end{equation}
	
	Moreover, we assume $\nu_\Omega=\nu_{\mathcal{C}_1}=\nu_{\mathcal{C}_2}=\la_\Om=1$ and fix the parameters $\alpha=1$, $\beta=1.5$, $\gamma=1.5$, $a=0.8$.
	
	Firstly, we perform some tests where we fix $b=c=0.5$ and we study the solutions for different values of $\la_{\mathcal{C}_1}$ and $\la_{\mathcal{C}_2}$:
	\begin{figure}[H]
		\centering
		\includegraphics[width=0.49\textwidth]{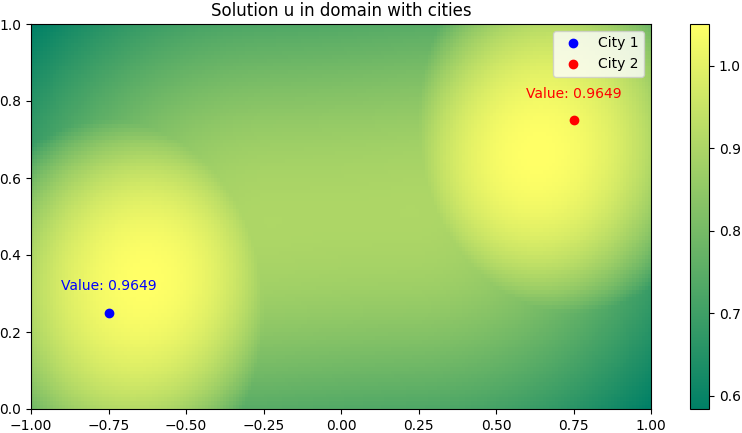}
		\hfill
		\includegraphics[width=0.49\textwidth]{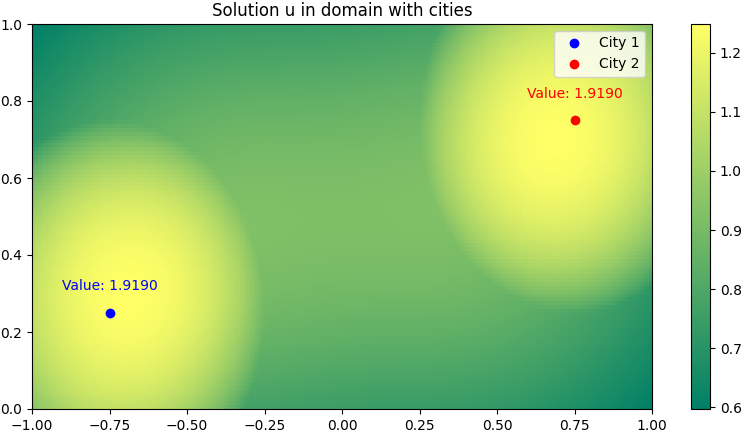}
		\caption{Solutions $u$, $\tilde u_1$ and $\tilde u_2$ of \eqref{ciudades_sist2} for $\la_{\mathcal{C}_1}=\la_{\mathcal{C}_2}=1$ and $\la_{\mathcal{C}_1}=\la_{\mathcal{C}_2}=2$.}
		\label{fig:continuous_example6}
	\end{figure}
	As in the previous example, the population concentrates around the two cities, even when  $\la_{\mathcal{C}_1}=\la_{\mathcal{C}_2}=\la_\Om=1$. Moreover, since $b=c$, $\beta=\ga$ and $\la_{\mathcal{C}_1}=\la_{\mathcal{C}_2}$, the solutions at both cities are identical, as expected.
	
	When $\la_{\mathcal{C}_1}=\la_{\mathcal{C}_2}=2$, the solutions in the cities increase, and the population density in $\Om$ surrounding the cities also rises. This reflects the higher influence of the cities, which provide more resources to the ecosystem.
	
	Now, we repeat the simulations for $\la_{\mathcal{C}_1}\neq\la_{\mathcal{C}_2}$, which represents the scenario where one city is more influential than the other.
	\begin{figure}[H]
		\centering
		\includegraphics[width=0.49\textwidth]{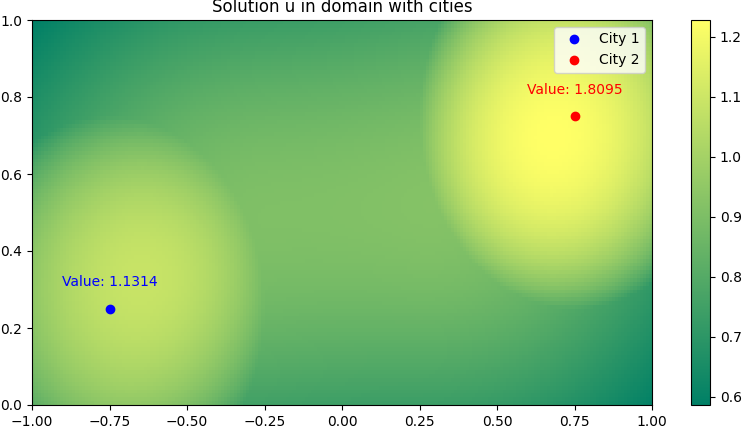}
		\hfill
		\includegraphics[width=0.49\textwidth]{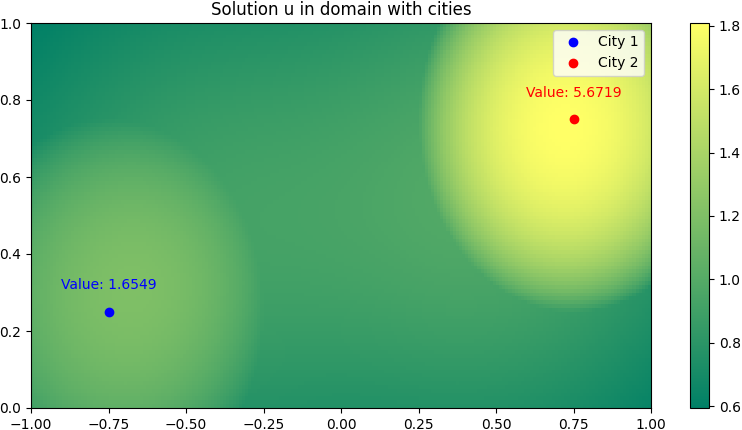}
		\caption{Solutions $u$, $\tilde u_1$ and $\tilde u_2$ of \eqref{ciudades_sist2} for $\la_{\mathcal{C}_2}=2$ when $\la_{\mathcal{C}_1}=1$ and $\la_{\mathcal{C}_1}=6$.}
		\label{fig:continuous_example7}
	\end{figure}
	As expected, the population density increases in the more influential city, while the less influential city supports a comparatively smaller population. The surrounding area in $\Om$ also reflects this contrast: regions near the stronger city exhibit higher population density, whereas regions near the weaker city remain less populated.
	
	Furthermore, we observe that if $\la_{\mathcal{C}_1}=6$, then $\tilde u_1$ is much larger than in the previous, since the contrast between both cities' influence is more noticeable. This also leads to a slightly larger value for $\tilde u_1$, since more individuals from city 2 can arrive to city 1.
	
	Finally, we fix $\la_{\mathcal{C}_1}=\la_{\mathcal{C}_2}=2$ and we compute the solutions for $b\neq c$. We highlight that these parameters represent the  maximum distance between points in the support of $k_2$ and $k_3$ relative to $\mathcal{C}_1$ and $\mathcal{C}_2$, respectively:
	\begin{figure}[H]
		\centering
		\includegraphics[width=0.49\textwidth]{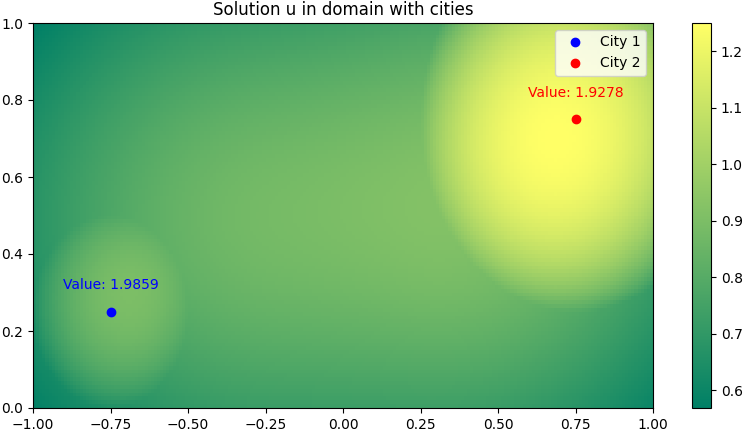}
		\hfill
		\includegraphics[width=0.49\textwidth]{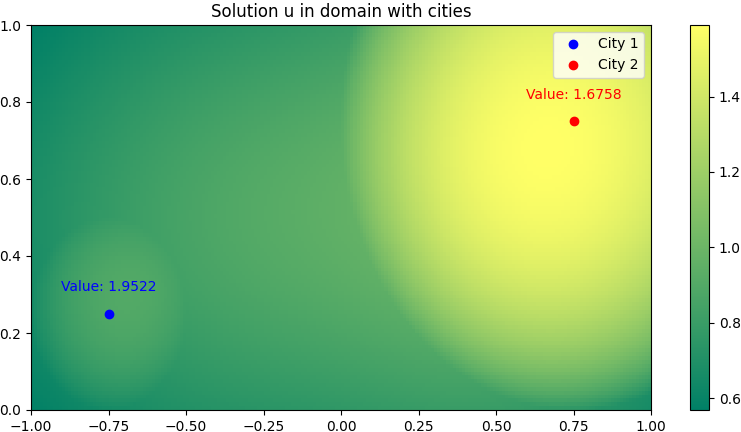}
		\caption{Solutions $u$, $\tilde u_1$ and $\tilde u_2$ of \eqref{ciudades_sist2} for $b=0.25$ when $c=0.5$ and $c=0.75$.}
		\label{fig:continuous_example8}
	\end{figure}
%	In both cases, $\tilde{u}_1$ is larger than when $b=0.5$, since the density of individuals in the surroundings of $\mathcal{C}_1$ decreases—there are fewer places to go—so they concentrate in the city. Another consequence is that individuals in $\Om$ who previously lived around $\mathcal{C}_1$ now inhabit the surroundings of $\mathcal{C}_2$. In said point, in the case $c=0.5$, the solution $\tilde u_2$ is slightly larger than in the case when $b=c=0.5$.
	
	However, for $c=0.75$, individuals tend to disperse more around $\mathcal{C}_2$. As a result, the density decreases in $\mathcal{C}_2$ itself, but increases within a neighborhood of radius $c$ around the city. This can model, for instance, towns located in the metropolitan area of a city that cannot sustain a higher population density.
\end{example}

In summary, these diverse numerical experiments successfully validate our theoretical findings for the logistic case, demonstrating the versatility of our framework to handle continuous, discrete, and highly complex hybrid configurations. This motivates the exploration of other types of restricted population growths within the same measure-theoretical setting, such as the classical model analyzed in the next subsection.

\subsection{Nonlocal Gompertz equation}
While the quadratic regulation of the logistic equation is widely used, many biological systems exhibit growth rates that slow down exponentially at high densities, or populations that are highly sensitive to overcrowding even at low saturation levels. To address these dynamics—frequently observed in restricted ecological habitats and tumor growth kinetics—we now verify that the nonlocal Gompertz growth model fits naturally into our framework. Specifically, we consider
 $\lambda\in L^\infty(\Omega)$ and $\nu\in L^\infty(\Omega)$ such that
 \beq\label{lamuestpos}
 \operatorname*{ess\,inf}_{x\in\Omega} \lambda(x):=\la_0 > 0,
 \qquad
 \operatorname*{ess\,inf}_{x\in\Omega} \nu(x) > 0.
 \eeq
 
 Let $F:\Omega\times[0,+\infty)\to\mathbb{R}$ be defined by
 \beq\label{defFGom}
 F(x,u)=
 \begin{cases}
 	\lambda(x)\,u\log\!\left(\dfrac{\nu(x)}{u}\right), & \text{if } u>0,\\[6pt]
 	0, & \text{if } u=0.
 \end{cases}
 \eeq
 This function has been continuously extended at $u=0$ so that $F$ is well defined and continuous on $[0,+\infty)$. Thus, it satisfies \eqref{caratheodory}.
 
 We consider the following problem of the form \eqref{geneq2}:
  \beq\label{gompertz}Lu=F(x,u),\quad u\geq0,\ \hbox{ a.e. in }\Om,\eeq
 which corresponds to the nonlocal Gompertz equation. In these conditions, we show there exists a unique strictly positive solution of \eqref{gompertz}.
 \begin{theorem}\label{TeoEGom}
 	There exists a unique strictly positive solution of \eqref{gompertz} in each $\Om_l\subset\Om$, $l\geq1$.
 \end{theorem}
 \begin{proof}
 	Our aim is to apply Theorem \ref{thexistence}. Clearly, \eqref{Fnoneg0} is satisfied since $F(x,0)=0$. Moreover, we have
 	$$\lim_{u\to\infty}\la u\log\left(\frac{\nu}{u}\right)=-\infty,\qquad\partial_u\left(\frac{F(x,u)}{u}\right)=-\frac{\la}{u}<0,\ \forall u>0.$$
 	Hence, \eqref{hipsupsol} and \eqref{hipunicidad} also hold.
 	
 	On other side, to prove \eqref{hipsubsol}, we consider $\alpha=\la_p(L)+1$, which satisfies $\la_p(L-\alpha I)<0$. Let $\ep>0$. We need to find $\delta(\ep)$ such that if $0<s\leq\delta$, then
 	$$F(x,s)\geq(\alpha-\ep)s\ \hbox{ a.e. }x\in\Om.$$
 	If $s=0$, it is trivial. For $s>0$, it suffices to show
 	$$\la_0\log\!\left(\dfrac{\nu(x)}{s}\right)\geq\alpha-\ep\Leftrightarrow\log\!\left(\dfrac{\nu(x)}{s}\right)\geq\frac{\alpha-\ep}{\la_0}.$$
 	Since
 	$$\lim_{s\to0^+}\log\left(\frac{\nu}{s}\right)=+\infty,$$
 	we have that $\forall M>0$ $\exists\tilde\delta(M)$ such that if $0<s<\tilde\delta$, then $\log\left(\nu/s\right)>M$. Therefore, taking $M>0$, $M>(\alpha-\ep)/\la_0$ and $\delta(\ep)=\tilde\delta(M)$, we conclude \eqref{hipsubsol}. We can finally apply Theorem \ref{thexistence} to deduce the result.
\end{proof}
As we did for the logistic equation, Theorem \ref{thregularity} can be applied to obtain the following regularity result
\begin{theorem}
	We assume \eqref{lamuestpos} and \eqref{defFGom} together with $\Om$ Hausdorff, locally compact, $\la\in C^0(\Om)$, $\nu\in C^0(\Om)$, $k\in C^0(\Om\times\Om)$ and $\mu$ a nonnegative Radon measure. We consider $\Om_l$, $l\geq 1$, the equivalence classes associated to the relation given by Definition \ref{defCEd1}. Then, in every $\Om_l\subset\Om$ such that there exists a positive solution $u$ in $\Om_l$, said solution satisfies $u\in C^0(\Om)$.
\end{theorem}
\begin{proof}Our aim is to apply Theorem \ref{thregularity}. 
The proof is similar to the one given for Theorem \ref{threglog}, but in this case
$$a(x)s-F(x,s)=\left\{\ba{l}
\dis a(x)s-\la(x) s\log\left(\frac{\nu(x)}{s}\right),\quad\hbox{if }s>0,\\
0,\quad\hbox{if }s=0.
\ea\right.$$
This function decreases in $\Big(0,\nu(x)e^{-\left(1+\frac{a(x)}{\la(x)}\right)}\Big)$ and increases in $\Big(\nu(x)e^{-\left(1+\frac{a(x)}{\la(x)}\right)},\infty\Big)$. Since it is zero in $s=0$, it is increasing and therefore injective, where it is positive. Condition \eqref{hipco2} is also easy to see taking into account that in every compact $K\subset\Om_l$ it holds
$$\ba{l}
as-\la s\log\left(\frac{\nu}{s}\right)=as-\la s\log(\nu)+\la s\log(s)\\\ecart\dis
\quad \qquad \qquad\qquad\geq\min_K (a)s-\max_K(\la\log(\nu))s+\min_K(\la)s\log(s)\to\infty,\ \hbox{ if }s\to\infty.
\ea$$
\end{proof}
\section*{Acknowledgements}
ACD, MMB and AS have been partially supported by Ministry of Science, Innovation and Universities of Spain under research project PID2023-149509NB-I00. AS also has been supported by University of Sevilla under grants SOL2024-31708 and SOL2024-31596.
The authors thank IMUS-María de Maeztu grant CEX2024-001517-M - Apoyo a
Unidades de Excelencia María de Maeztu for supporting this research, funded by
MICIU/AEI/ 10.13039/501100011033
\end{document}